\documentclass[12pt,reqno]{amsart}

\usepackage{etex}
\usepackage[utf8]{inputenc}
\usepackage{amsfonts}
\usepackage{amsmath}
\usepackage{amssymb}
\usepackage{amsthm}
\usepackage{mathrsfs}
\usepackage{stmaryrd}
\usepackage{color}
\usepackage[english]{babel}
\usepackage[T1]{fontenc}
\usepackage{url}
\usepackage{graphicx}
\usepackage[pdftex,colorlinks,backref=page,citecolor=blue]{hyperref}
\usepackage{caption}
\usepackage{enumerate}
\usepackage{epstopdf}
\usepackage{hyphenat}
\usepackage{float}
\usepackage{indentfirst}
\usepackage[export]{adjustbox}
\usepackage{tikz}
\usepackage[all]{xy}
\usetikzlibrary{matrix,arrows,decorations.pathmorphing}
\usepackage{booktabs}
\usepackage{array}
\usepackage{bm}
\usepackage{csquotes}
\usepackage{enumerate}
\usepackage{comment}
\usepackage{multicol}
\usepackage[margin=1 in,marginpar=0.75in]{geometry}
\usepackage{todonotes}

\newcolumntype{P}[1]{>{\centering\arraybackslash}p{#1}}

\allowdisplaybreaks

\newtheorem{theorem}{Theorem}[section]

\newtheorem{lemma}[theorem]{Lemma}
\newtheorem{proposition}[theorem]{Proposition}
\newtheorem{corollary}[theorem]{Corollary}
\newtheorem{theoremx}{Theorem}

\theoremstyle{definition}

\newtheorem{remark}[theorem]{Remark}
\newtheorem{notation}[theorem]{Notation}
\newtheorem{example}[theorem]{Example}
\newtheorem{definition}[theorem]{Definition}
\newtheorem{question}[theorem]{Question}

\newcommand{\N}{\mathbb{N}}
\newcommand{\Q}{\mathbb{Q}}
\newcommand{\KK}{\Bbbk}

\newcommand{\HK}{\mathrm{HK}}

\newcommand{\FS}{\mathrm{FS}}
\newcommand{\FSS}{\mathrm{FSS}}
\newcommand{\fs}{\mathrm{s}}
\def\u{{\bf u}}
\def\v{{\bf v}}
\def\a{{\bf a}}
\def\b{{\bf b}}
\def\c{{\bf c}}

\DeclareMathOperator{\frk}{frk}

\makeatletter
\def\l@subsection{\@tocline{2}{0pt}{2.5pc}{5pc}{}}
\makeatother

\makeatletter
\def\author@andify{%
	\nxandlist {\unskip ,\penalty-1 \space\ignorespaces}%
	{\unskip {} \@@and~}%
	{\unskip \penalty-2  ,~}%
}
\makeatother

\author{Alessio Caminata}
\address{Alessio Caminata, Dipartimento di Matematica, Universit\`a di Genova\\ via Dodecaneso 35, 16146, Genova, Italy}
\email{alessio.caminata@unige.it}

\author{Samuel Shideler}
\address{Samuel Shideler}
\email{sjshide@gmail.com}

\author{Kevin Tucker}
\address{Kevin Tucker, Department of Mathematics, Statistics, and Computer Science, University of Illinois Chicago, 851 S. Morgan Street, 322 Science and Engineering Offices (MC 249), Chicago, IL 60607-7045, USA}
\email{kftucker@uic.edu}

\author{Francesco Zerman}
\address{Francesco Zerman, Dipartimento di Matematica, Universit\`a di Genova\\ via Dodecaneso 35, 16146, Genova, Italy}
\email{francesco.zerman@edu.unige.it}

\title[F-signature functions of diagonal hypersurfaces]{F-signature functions of diagonal hypersurfaces}


\begin{document}
	
	\maketitle
	
	\begin{abstract}
		Let $f$ be a diagonal hypersurface in $A_p=\mathbb{F}_p\llbracket x_1,\dots,x_n\rrbracket$. We study the behavior of the function $\phi_{f,p}({a}/{p^e})=p^{-ne}\dim_{\mathbb{F}_p}\big(A_p/(x_1^{p^e},\dots,x_n^{p^e},f^a)\big)$ which encodes information about the $F$-threshold, the Hilbert-Kunz, and the $F$-signature functions. We prove that when $p$ goes to infinity $\phi_{f,p}$ converges to a piecewise polynomial function $\phi_f$ and the left and right derivatives of $\phi_{f,p}$ converge to $\phi'_f$. We use this fact to prove the existence of the limit $F$-signature and limit Hilbert-Kunz multiplicity for diagonal hypersurfaces. When $f$ is a Fermat hypersurface, we investigate the shape of the $F$-signature function of $f$ and provide an explicit formula for the limit $F$-signature and, in some cases, also for the $F$-signature for fixed $p$. This allows us to answer negatively to a question of Watanabe and Yoshida.
	\end{abstract}

	\section{introduction}
	Let $A_p=\mathbb{F}_p\llbracket x_1,\dots,x_n\rrbracket$ be the power series ring in $n$ variables over the finite field $\mathbb{F}_p$ of positive characteristic $p$, and let $I$ be an ideal of $A_p$ such that $R=A_p/I$ is reduced. The {\it $F$-signature function of $R$} is the numerical function
	\[
	\FS_R(e)=\frk_R(R^{1/p^e})
	\]
	where $R^{1/p^e}$ is the ring obtained by adjoining $p^e$-th roots of elements in $R$, and $\frk_R(R^{1/p^e})$ denotes the maximal rank of a free $R$-summand of $R^{1/p^e}$.
	The $F$-signature function was introduced by Smith and Van den Bergh \cite{SVdB97} in the context of rings with finite F-representation type, and later considered in greater generality by Huneke and Leuschke \cite{HunLeu}, who introduced the name $F$-signature, and by Watanabe and Yoshida \cite{WatYosh}, under the name of ``minimal relative Hilbert-Kunz multiplicity.'' The third author   (see \cite{Tucker12, PolTuc}) proved that 
	\[
	\FS_R(e)=\fs(R)p^{e\dim R}+O(p^{(\dim R-1)e}),
	\]
	where $\dim R$ is the Krull dimension of the ring and the leading coefficient $\fs(R)$ is a real number in the interval $[0,1]$, called the \emph{$F$-signature}. This invariant encodes information about the singularities of the ring. In fact, we have that $\fs(R)=1$ if and only if $R$ is regular \cite{HunLeu, WatYos2000} and $\fs(R)>0$ if and only if $R$ is strongly $F$-regular \cite{AL03}.
	On the other hand, $\fs(R)$ is typically very hard to compute and the precise value is known only for certain limited classes of rings. The function  $\FS_R(e)$ is even more mysterious and difficult to predict, with a few very special exceptions \cite{CDS19, Sin05}.

	When $I=(f)$ is a hypersurface, we can introduce and study a more general function $\phi_{f,p}$ which goes from $\mathscr{I}=[0,1]\cap\left\{\frac{a}{p^e}: \ a,e\in\N\right\}$ to $\Q$ and is defined as
	\[
	\phi_{f,p}\left(\frac{a}{p^e}\right)=\frac{1}{p^{{ne}}}\dim_{\KK}\big(A_p/(x_1^{p^e},\dots,x_n^{p^e},f^a)\big).
	\]
	The function $\phi_{f,p}$ was first introduced by Monsky and Teixeira \cite{MonTexI,MonTexII}, and later studied by the third author, Blickle, and Schwede \cite{BST13} who proved that $\phi_{f,p}$ can be extended to a Lipschitz continuous increasing concave function defined on the whole interval $[0,1]$, and that $1-\phi_{f,p}$ coincides with $F$-signature of pairs. We also mention that $\phi_{f,p}(0)=0$, $\phi_{f,p}(1)=1$. Moreover, they prove that the left derivative of $\phi_{f,p}$ exists at $t = 1$ and equals the $F$-signature of $A_p/f$. Similarly, the right derivative of $\phi_{f,p}$ exists at $t = 0$ and equals the Hilbert–Kunz multiplicity of $A_p/f$ (see \cite[Theorem~4.4]{BST13}). 
	One of the main interests of $\phi_{f,p}$ relies in the fact that it encodes at the same time information about the $F$-signature function and the Hilbert-Kunz function, which is another important numerical function that can be defined in the same setup and was first introduced by Kunz  \cite{Kun69, Kun76}. In a nutshell, the values of $\phi_{f,p}$ at points $\frac{1}{p^{e}}$ give the Hilbert-Kunz function, and the values of $\phi_{f,p}$ at points $1-\frac{1}{p^{e}}$ give the $F$-signature function (see Lemma~\ref{lemma_$F$-signaturefractal} for a precise statement). Moreover, the intersection of $\phi_{f,p}$ with the line $y=1$ gives the  $F$-pure threshold \cite{BST13}. 
	On the other hand, we point out that often the function $\phi_{f,p}$ is much more difficult to compute than both the $F$-signature and the Hilbert-Kunz functions. For example, for the cubic hypersurface  $f=x^3+y^3+xyz$ in characteristic $2$, both the $F$-signature and the Hilbert-Kunz functions are known (see \cite{BuchC97}). On the other hand, the shape of the function $\phi_{f,p}$ is only conjecturally known and is related to Monsky's conjecture about the irrationality of the Hilbert-Kunz multiplicity of the hypersurface $f+uv$ (see \cite{MonAlg} for more details). Finally, we point out that similar functions encoding information about the Hilbert-Kunz multiplicity of hypersurfaces in positive characteristic have been studied also by Chiang and Hung \cite{CH98}, Meng and Mukhopadhyay \cite{MM23}, Ohta \cite{Ohta17}, and Trivedi \cite{Trivedi18, Trivedi23}.
	
	The aim of this article is to investigate the function $\phi_{f,p}$ for diagonal hypersurfaces, that is when $f=x_1^{d_1}+\cdots+x_n^{d_n}$ for some integers $d_1,\dots,d_n\geq2$.
	We have two main goals. On the one hand, we study the behaviour of $\phi_{f,p}$ for $p\to\infty$ and we prove that it converges to a piecewise polynomial function. On the other hand, we investigate explicitely the functions $\phi_{f,p}$ (for fixed $p$) to study the shape of the $F$-signature function $\FS_{A_p/f}(e)$ of Fermat hypersurfaces.
	
	Our first main result is the following.
	
	\begin{theoremx}[Theorem~\ref{thm_15} and Theorem~\ref{thm16}]\label{thmA}
		Let $2\leq d_1\leq \cdots\leq d_n$ be integers and let $f=x_1^{d_1}+\cdots+x_n^{d_n}$.
		Then the functions $\phi_{f,p}$ converge uniformly, as $p$ goes to infinity, to a piecewise polynomial function $\phi_f$ given by
		\[
		\phi_f(t)=\frac{d_1\cdots d_n}{2^n\cdot n!}\left(C_0(t)+2\sum_{\lambda\in\mathbb{Z}_{\geq1}}C_{\lambda}(t)\right),
		\]
		where $C_{\lambda}(t)$ for $\lambda\in\mathbb{Z}_{\geq0}$ is the piecewise polynomial function
		\[
		C_{\lambda}(t)=\sum(\epsilon_0\cdots\epsilon_n)\left(\epsilon_0t+\frac{\epsilon_1}{d_1}+\cdots+\frac{\epsilon_n}{d_n}-2\lambda\right)^n
		\]
		with the sum taken over all choices of $\epsilon_0,\dots,\epsilon_n\in\{\pm1\}$ with $\epsilon_0t+\frac{\epsilon_1}{d_1}+\cdots+\frac{\epsilon_n}{d_n}-2\lambda\geq0$.
		Moreover, for any $t\in[0,1]$, the sequence of derivatives $\partial_{-}\phi_{f,p}(t)$ converges to $\partial_{-}\phi_f(t)$ and $\partial_{+}\phi_{f,p}(t)$ converges to $\partial_{+}\phi_f(t)$. In particular, for all but finitely many explicit points, both sequences converge to $\phi_f'(t)$.
	\end{theoremx}
	
	The previous theorem enables us to compute the piecewise polynomial function $\phi_f(t)$ for simple diagonal hypersurfaces such as $f=x^2+y^3$ (Example~\ref{ex-x2y3}) or some\footnote{Precisely, the ADE singularities which are diagonal have types $A_n$, $E_6$, and $E_8$.} two-dimensional ADE singularities (Remark~\ref{remarkADE}). 
	Moreover, with this result at our disposal, we can use the connection between $\phi_{f,p}$ and Hilbert-Kunz and $F$-signature functions to obtain several interesting corollaries. We prove that the limits $\lim_{p\to\infty}e_{\HK}(A_p/f)$ and $\lim_{p\to\infty}\fs(A_p/f)$ exist and we provide combinatorial formulas expressing the values of those limits (see Corollary~\ref{cor:limitHK} and Corollary~\ref{cor:Fsignaturediagonal}).
	Moreover, for the Fermat quadric, i.e. when $d_1=\cdots=d_n=2$, we obtain nicer formulas for the limit function $\phi_f=\lim_{p\to\infty}\phi_{f,p}$ in terms of Euler polynomials (see Theorem~\ref{thm:squareslim}). These formulas allow us to recover a result by Gessel and Monsky \cite{GesselMonsky} who showed, with a different proof, that the limit value $\lim_{p\to\infty}e_{\HK}(A_p/f)= 1+ c_n$, where $c_n$ is the coefficient of $z^{n-1}$ in the power series expansion of  $\sec(z)+\tan(z)$.
	
	In the second part of the paper, we focus on the computation of $\FS_{A_p/f}$ for the Fermat hypersurface $f=x_1^d+\cdots+x_n^d$. The most interesting case is when $p>d$ and $n>d$, otherwise the quotient $A_p/f$ is not strongly $F$-regular. We are able to prove the following. 
	
	\begin{theoremx}[Theorem~\ref{thm:shapeFermat}]
		Let $n>d\ge2$ be integers with $p>d$ and $p\equiv\pm1\ \mathrm{mod}\ d$ and assume that 
		\begin{equation*}
			\begin{cases}
				n\cdot\left(\frac{p-1}{d}\right)  &\text{is even if $p\equiv 1\ \mathrm{mod}\ d$}\\
				n\cdot\left(\frac{p+1}{d}\right)  &\text{is even if $p\equiv -1\ \mathrm{mod}\ d$}.
			\end{cases}
		\end{equation*}
		Let $f=x_1^d+\cdots+x_n^d$, then there exist constants $B,C\in\mathbb{Z}$ (depending on $p,d,n$) with $0\le B\le p^{n-3}$ such that for every $e\geq1$
		\[
		\FS_{A_p/f}(e)= \fs(A_p/f)\cdot p^{(n-1)e}+(1-\fs(A_p/f))\cdot B^e 
		\]
		with $\fs(A_p/f)=\displaystyle\frac{-C}{p^{n-1}-B}\in\mathbb{Q}$. 
	\end{theoremx}
	
	Determining the precise value of the constants $B$ and $C$, and thus of the $F$-signature $\fs(A_p/f)$, can be complicated. The first non-trivial case is when $n=d+1$, that is the Fermat hypersurface of degree $d$ in $d+1$ variables. In this case, Watanabe and Yoshida \cite{WatYos05} proved that the $F$-signature is always smaller or equal than the value $\frac{1}{2^{d-1}(d-1)!}$, and asked whether this is in fact an equality. We are able to provide a negative answer to this question in the case $d=3$ by computing the exact value of  $\fs(A_p/f)$, which depends on $p$, but is always smaller than the Watanabe--Yoshida upper bound (see Theorem~\ref{thm:FsignatureFermatcubic}).
	On the other hand, building up on the formula of Theorem~\ref{thmA}, we can prove that when $p$ grows to infinity the $F$-signature limits to the known upper bound (see Proposition~\ref{cor:WYupperboundlimit}).
	
	In the last part of the paper, we present explicit examples of $F$-signature functions of hypersurfaces.
	This shows that the behaviour of the function $\FS_{A_p/f}(e)$ can be erratic and depends on the characteristic of the base field. For example, for the Fermat cubic in four variables $f=x_1^3+x_2^3+x_3^3+x_4^3$, the $F$-signature function is a polynomial in $p^e$ when $p=5$, but for $p=7$ it is not polynomial or quasi-polynomial, being $\FS_{A_p/f}(e)=\frac{21}{170} 7^{3e} + \frac{149}{170} 3^e$. A general explicit formula for any $p$ can be found in Theorem~\ref{thm:FsignatureFermatcubic}.
	
	\subsection*{Structure of the paper}
	In Section~\ref{section:preliminaries} we collect several preliminary results about the $F$-signature and the functions $\phi_{f,p}$ which will be used in the rest of the paper. Section~\ref{section:limitphifunctions} is dedicated to the proof of Theorem~\ref{thmA} about the limit function $\phi_f$ and its corollaries. In Section~\ref{section:Fermatquadric}, we focus on the Fermat quadric and the computation of $\phi_f$ in terms of Euler polynomials. Section~\ref{section:fsignatureseries} contains some technical results which will be needed during the computation of the $F$-signature function of the Fermat hypersurface, which is carried out in Section~\ref{section:Fermat}. Finally, in Section~\ref{section:examples} we present some examples of $F$-signature functions.

	\subsection*{Notation}
	For simplicity, whenever we are interested in the behaviour for $p\to\infty$, we consider the diagonal hypersurface $f=x_1^{d_1}+\cdots+x_n^{d_n}$ as an element of power series ring  $A_p=\mathbb{F}_p\llbracket x_1,\dots,x_n\rrbracket$  over finite fields $\mathbb{F}_p$ for different values of $p$.
	On the other hand, when we are interested in computations for a fixed $p$, we can relax this assumption and work over any perfect field $\KK$ of characteristic $p$. To avoid confusion, at the beginning of each section we state clearly our assumptions.
	
	Sometimes, we may also denote the $F$-signature $\fs(A_p/f)$ and the $F$-signature function $\FS_{ A_p/f}(e)$ of a hypersurface $f$ simply by $\fs(f)$ and $\FS_{f}(e)$ respectively. 
	
	\subsection*{Acknowledgements}
	
	We would like to thank Luis N\'{u}\~{n}ez-Betancourt, Holger Brenner, Alessandro De Stefani, Neil Epstein, Yusuke Nakajima for several useful discussions.
	
	A.~Caminata is supported by the Italian PRIN2020 grant 2020355B8Y ``Squarefree Gr\"obner degenerations, special varieties and related topics'',  by the Italian PRIN2022 grant P2022J4HRR ``Mathematical Primitives for Post Quantum Digital Signatures'', by the INdAM--GNSAGA grant ``New theoretical perspectives via Gr\"obner bases'', by the MUR Excellence Department Project awarded to Dipartimento di Matematica, Università di Genova, CUP D33C23001110001, and by the European Union within the program NextGenerationEU. Additionally, part of the work was done while Caminata was visiting the Institute of Mathematics of the University of Barcelona (IMUB). He gratefully appreciates their hospitality during his visit.
	
	K.~Tucker is partially supported by National Science Foundation Grants DMS \#2200716 and \#2501904, and by the Simons Foundation Travel Support for Mathematicians SFI-MPS-TSM-00014083. Additionally, part of this research was performed while Tucker was visiting the Mathematical Sciences Research Institute (MSRI), now the Simons Laufer Mathematical Sciences Institute (SLMath), which was supported by the National Science Foundation Grant DMS \#1928930.

	
	\section{Preliminaries}\label{section:preliminaries}
	For the convenience of the reader, we recall some notation and facts from \cite{HanMon, MonTexI,MonTexII}, together with some preliminary results. We work over a perfect field $\KK$ of positive characteristic $p$.
	\subsection{Representation ring}
	We denote by $\Gamma$ the Grothendieck group of the semigroup of isomorphism classes of finitely generated $\KK[T]$-modules on which the variable $T$ acts nilpotently. Such modules are called $\KK$-objects. The sum in $\Gamma$ is given by direct sum, and we can also equip $\Gamma$ with a (commutative) ring structure using the following product: if $M$ and $N$ are two $\KK$-objects, the class of $M\otimes_\KK N$ represents the product of the classes of $M$ and $N$ in $\Gamma$. The action of $T$ on   $M\otimes_\KK N$ is given by $T(a\otimes b)=(Ta)\otimes b + a\otimes(Tb)$. The ring $\Gamma$ is called the \emph{representation ring}. The class of the zero module, which will be denoted by $\delta_0$, is the zero element of the ring, and the class of the module $\KK[T]/(T)$, which will be denoted by $\delta_1$, is the neutral element of the product. More generally, for any $e\in\mathbb{N}$, we denote by $\delta_e$ the image of the module $\KK[T]/(T^e)$ in $\Gamma$. Then, as a group, $\Gamma$ is the free Abelian group generated by the basis $\{\delta_i: \ i>0\}$. Another basis is given by $\{\lambda_i: \ i\in\mathbb{N}\}$, where $\lambda_i=(-1)^i(\delta_{i+1}-\delta_i)$. We point out that the $\lambda_i$-coordinate (of the image in $\Gamma$) of a $\KK$-object $M$ is  $(-1)^i\dim_\KK(T^iM/T^{i+1}M)$.
	For any $e\in\mathbb{N}$, we define $\Gamma_e$ to be the additive subgroup of $\Gamma$ generated by $\{\lambda_i: \ i<p^e\}$. It turns out that $\Gamma_e$ is a subring of $\Gamma$ (see \cite[Theorem~3.2]{HanMon}).
	
	The two bases $\{\delta_i\}$ and $\{\lambda_i\}$ of $\Gamma$ satisfy several multiplication rules which are worked out in \cite{HanMon}. We collect some of them here. We shall use them repeatedly in our computations.
	\begin{align}\label{eq:formule moltiplicazione lambda}
		\delta_i\delta_{p^ej}&=i\cdot\delta_{p^ej}   &&(0\leq i\leq p^e, \ j\geq1) \nonumber \\
		\lambda_i\lambda_{p^ej}&=\lambda_{p^ej+i}  &&(0\leq i< p^e)\nonumber\\
		\lambda_i\lambda_{p^e j-1}&=\lambda_{p^e j-1-i}  &&(0\leq i< p^e) \\
		\lambda_i\lambda_j&=\sum_{k=j-i}^{\min\{i+j,2p-2-i-j\}}\lambda_k &&(0\leq i\le j< p) . \nonumber
	\end{align}
	We will also need an estimate on the coefficients of the powers of the $\lambda_i$'s. In particular, if $r\ge 2$ and $0\le i\le p-1$ then \cite[Theorem 2.7]{HanMon} yields that
	\begin{equation}\label{eq:stima coefficienti potenza lambda}
		\lambda_i^r=\sum_{j=0}^{p-1}c_j\lambda_j \quad \text{with} \quad 0\le c_j\le p^{r-2}.
	\end{equation}
	Following \cite[§2]{MonTexII}, we define two $\mathbb{Z}$-linear functions $\alpha: \Gamma\longrightarrow\mathbb{Z}$ and $\theta: \Gamma\longrightarrow\Gamma$ as follows 
	\vspace{-10mm}
	\begin{multicols}{2}
		\[
		\begin{split}
			\alpha\left(\sum_{i=0}^nc_i\lambda_i\right)= c_0,
		\end{split}
		\]
		\break
		\[
		\begin{split}
			\theta(\lambda_i)= \begin{cases} \lambda_{pi} \ &\text{ if } i \text{ is even}\\
				\lambda_{pi+p-1} \ &\text{ if } i \text{ is odd.}
			\end{cases}
		\end{split}
		\]
	\end{multicols}
	The map $\alpha$ has the property that $\alpha(\lambda_i\lambda_j)=1$ if $i=j$ and $\alpha(\lambda_i\lambda_j)=0$ otherwise.  
	The operator $\theta$ is actually a ring homomorphism by \cite[Theorem~2.13]{MonTexII} and satisfies $\theta^{e}(\Gamma)\Gamma_e=\Gamma$ for all $e>0$.
	Since we will need to work with rational coefficients, we  consider the ring $\Gamma_{\mathbb{Q}}=\Gamma\otimes_{\mathbb{Z}}\mathbb{Q}$ and extend the maps $\alpha$ and $\theta$ in the natural way as $\alpha\otimes \mathrm{Id}:\Gamma_{\mathbb{Q}}\rightarrow \mathbb{Q}$ and $\theta\otimes \mathrm{Id}:\Gamma_{\mathbb{Q}}\rightarrow\Gamma_{\mathbb{Q}}$. We shall denote these maps simply by $\alpha$ and $\theta$ respectively.
	
	Recalling that the $\lambda_0$-coordinate of a $\KK$-object $M$ is just the $\KK$-vector space dimension of $M/TM$, we can write $\alpha(M)=\dim_\KK M/TM$ and we can generalize this notion defining $\alpha_k(M)=\dim_\KK M/T^kM$.
	
	In order to relate this to length computations of diagonal hypersurfaces, we consider, as in \cite[§1]{HanMon}, the number
	\[
	D_{\KK}(k_1,\dots,k_n) = \dim_{\KK}\left(\KK[x_1,\dots,x_n]/(x_1^{k_1},\dots,x_n^{k_n},x_1+\cdots+x_n)\right)
	\]
	for every nonnegative integers $k_1,\dots, k_n$. If the variable $T$ acts as multiplication by $x_i$ on $\delta_{k_i}=\KK[x_i]/x_i^{k_i}$, then $T$ acts as multiplication by $x_1+\dots+x_n$ on the product $\prod_i\delta_{k_i}=\KK[x_1,\dots,x_n]/(x_1^{k_1},\dots,x_n^{k_n})$. Thus we obtain that
	\begin{equation*}
		D_\KK(k_1,\dots,k_n)=\alpha\left(\prod_{i=1}^n\delta_{k_i}\right),
	\end{equation*}
	as remarked also in \cite[Theorem 1.9]{HanMon}.
	Moreover, for $t \in \mathbb{Q}_{> 0}$ we define 
	\[
	\delta_t := (1-t+\lfloor t \rfloor)\delta_{\lfloor t \rfloor} + (t - \lfloor t \rfloor)\delta_{\lfloor t + 1 \rfloor} \in \Gamma_{\mathbb{Q}}.
	\]
	Alternatively, writing $t = \frac{a}{b}$ for positive integers $a,b$, one checks $\delta_{\frac{a}{b}} \in \Gamma_{\mathbb{Q}}$ is equal to $\frac{1}{b}$ times the class of the $\KK$-object $\KK[x]/(x^a)$ where $T$ acts as multiplication by $x^b$.
	Similarly as before, the product of elements $\delta_{\frac{a_1}{b_1}}\cdots\delta_{\frac{a_n}{b_n}}$ is then $\frac{1}{b_1\cdots b_n}$ times the class of the $\KK$-object $\KK[x_1,\dots,x_n]/(x_1^{a_1},\dots,x_n^{a_n})$ where $T$ acts as multiplication by the diagonal hypersurface $x_1^{b_1}+\cdots+x_n^{b_n}$.

	\subsection{$F$-signature and Hilbert-Kunz of hypersurfaces}
	Let $A_p=\KK\llbracket x_1,\dots,x_n\rrbracket$ be the power series ring in $n$ variables and consider an element $f$ of the maximal ideal of $A_p$ such that the quotient ring $R=A_p/f$ is reduced. The {\it $F$-signature function of $R$} is the numerical function
	\[
	\FS_R(e)=\frk_R(R^{1/p^e})
	\]
	where $R^{1/p^e}$ is the ring obtained by adjoining $p^e$-th roots of elements in $R$, and $\frk_R(R^{1/p^e})$ denotes the maximal rank of a free $R$-summand of $R^{1/p^e}$ or, equivalently, the maximal rank of a free $R$-module $P$ for which there is a surjection $R^{1/p^e} \to P\to 0$.
	It is known (see \cite{Tucker12, PolTuc}) that 
	\[
	\FS_R(e)=\fs(R)p^{(n-1)e}+O(p^{(n-2)e}),
	\]
	where $n-1$ is the dimension of $R$ and the leading coefficient $\fs(R)$ is a real number in the interval $[0,1]$, called \emph{$F$-signature}.
	The \emph{Hilbert-Kunz function of $R$} is the function 
	\[
	\HK_R(e)=\dim_{\KK}\left(R/\mathfrak{m}^{[p^e]}\right),
	\]
	where $\mathfrak{m}^{[p^e]}$ denotes the Frobenius power of the maximal ideal, generated by all the $p^e$-th powers of elements of $\mathfrak{m}$. Also for this function, it is known (\cite{Mon83}) that  
	\[
	\HK_R(e)=e_{\HK}(R)p^{(n-1)e}+O(p^{(n-2)e}),
	\]
	where $e_{\HK}(R)$ is a positive real number called \emph{Hilbert-Kunz multiplicity}.
	
	The $F$-signature and Hilbert-Kunz functions of a hypersurface can be treated together by  associating with the element $f$ an element  $\phi_{f,p}$ of the space $\Q^{\mathscr{I}}$ 	of functions from $\mathscr{I}=[0,1]\cap\left\{\frac{a}{p^e}: \ a,e\in\N\right\}$ to $\Q$ 
	
	\begin{equation}\label{eq:defphi_p,f}
		\begin{split}
			\phi_{f,p}: \  &\mathscr{I}\longrightarrow\Q\\
			&\frac{a}{p^e}\longmapsto \frac{1}{p^{{ne}}}\dim_{\KK}\big(A/(x_1^{p^e},\dots,x_n^{p^e},f^a)\big).
		\end{split}
	\end{equation}
	
	The main connection between the function  $\phi_{f,p}$ and the Hilbert-Kunz and $F$-signature functions is expressed in the following lemma.

	\begin{lemma}\label{lemma_$F$-signaturefractal}
		For any $f\in A$ and any $e\in\N$, we have
		\begin{enumerate}
			\item $\HK_{R}(e)= p^{ne}\phi_{f,p}\left(\frac{1}{p^e}\right)$
			\item $\FS_{R}(e)=p^{ne}\left(1-\phi_{f,p}\left(1-\frac{1}{p^e}\right)\right)$.
		\end{enumerate}
	\end{lemma}
	\begin{proof}
		Point (1) follows immediately from the definition and point (2) is \cite[Proposition~4.1]{BST13}.
	\end{proof}

	Functions of the form $\phi_{f,p}$ where first studied by  Monsky and Teixeira \cite{MonTexI,MonTexII} who introduced the notion of $p$-fractal. A function $\phi: \mathscr{I}\longrightarrow\Q$  is a $p$-fractal if it lies in a finitely dimensional space stable under the action of the operators $T_{p^e|b}$ for any $e\in\mathbb{N}$ and integer $0\leq b< p^e$, which are defined as
	\[
	(T_{p^e|b}\phi)\left(\frac{a}{p^m}\right)=\phi\left(\frac{a+bp^m}{p^{m+e}}\right).
	\]
	Monsky and Teixeira proved that if $\phi_{f,p}$ is a $p$-fractal, then $f$ has rational Hilbert-Kunz series. 
	\begin{remark}\label{rk:analytic properties phi and psi}
		The functions $\phi_{f,p}$ have been studied also in \cite{BST13}. There, Blickle, Schwede and the third author introduce the function
		\begin{equation}\label{eq:psi}
			\psi_{f,p}\left(\frac{a}{p^e}\right)=1-\phi_{f,p}\left(\frac{a}{p^e}\right)=\frac{1}{p^{{ne}}}\dim_{\KK}\big(A/(\mathfrak{m}^{[p^e]}:f^a)\big)
		\end{equation}
		and prove that $\psi_{f,p}$ can be extended to a convex, Lipschitz continuous \cite[Corollary~3.3 and Theorem~3.5]{BST13} and decreasing (since $\psi_{f,p}(0)=1$, $\psi_{f,p}(1)=0$ and $\psi_{f,p}$ is non-negative) function defined on the whole interval $[0,1]$. These facts imply that the function $\phi_{f,p}$ can be extended to a concave, increasing, Lipschitz continuous function in the interval $[0,1]$. Moreover, they prove that the left derivative of $\phi_{f,p}$ exists at $t = 1$ and equals the $F$-signature of $A_p/f$. Similarly, the right derivative of $\phi_{f,p}$ exists at $t = 0$ and equals the Hilbert–Kunz multiplicity of $A_p/f$ (see \cite[Theorem~4.4]{BST13}). 
	\end{remark}

	\subsection{Sequences in $\Gamma$ and their operators}
	
	In this subsection we recall some results from \cite[§3]{MonTexII}. Denote by  $\Lambda=(\Gamma_{\mathbb{Q}})^{\mathbb{N}}$ the ring of sequences $\u=(u_0,u_1,\dots)$ with entries in $\Gamma_{\mathbb{Q}}$ and componentwise sum and multiplication. To any element $\phi$ of the space $\Q^{\mathscr{I}}$ 	of functions from $\mathscr{I}=[0,1]\cap\left\{\frac{a}{p^e}: \ a,e\in\N\right\}$ to $\Q$ we can associate an element $\mathscr{L}$ of $\Lambda$ whose $e$-th entry is
	\begin{equation}
		\label{eq:definizione serie L}
		\mathscr{L}(\phi)_e=\sum_{i=0}^{p^e-1}\left(\phi\left(\frac{i+1}{p^e}\right)-\phi\left(\frac{i}{p^e}\right)\right)(-1)^i\lambda_i.
	\end{equation}

	So we obtain a map $\mathscr{L}:\Q^{\mathscr{I}}\rightarrow \Lambda$, whose kernel is the $1$-dimensional subspace of constant functions. We denote the image of the map $\mathscr{L}$ by $\Lambda_0=\mathscr{L}(\Q^{\mathscr{I}})$, that is actually a $\mathbb{Q}$-subalgebra of $\Lambda$. An alternative description of $\Lambda_0$ can be found in \cite[\S3]{MonTexII}.

	We endow $\Lambda$ with a $\Gamma$-module structure by introducing the following scalar product 
	\[
	w\cdot\u=\big(wu_0,\theta(w)u_1,\theta^2(w)u_2,\dots\big) 
	\]
	for any $w\in\Gamma$ and $\u=(u_0,u_1,\dots)\in\Lambda$.
	
	We also introduce two maps $R,S:\Lambda\rightarrow\Lambda$ as follows. For any $\u=(u_0,u_1,\dots)\in\Lambda$ we have
	\begin{equation}\label{eq:definition reflection and shifting}
		\begin{split}
			R(\u)&=(v_0,v_1,\dots) \ \ \text{ where } \ \ v_e=(-1)^{p^e}\lambda_{p^{e}-1}u_e;\\
			S(\u)&=(u_1,u_2,\dots).
		\end{split}
	\end{equation}
	
	The map $R$ is $\Gamma$-linear and is called \emph{reflection}, the map $S$ is called \emph{shift operator}. The name reflection is motivated by the fact that 
	\[
	\mathscr{L}(\overline{\phi})=R(\mathscr{L}(\phi))
	\]
	for any $\phi\in\Q^{\mathscr{I}}$, where $\overline{\phi}$ is the \emph{reflection}  of $\phi$ defined as
	\begin{equation}\label{eq:reflection}
		\bar{\phi}(t)=\phi(1-t)
	\end{equation}
	for every $t\in\mathscr{I}$.
	The formulas in \eqref{eq:formule moltiplicazione lambda} imply also that $R^2$ is the identity operator on $\Lambda$, and that
	\begin{equation}\label{eq:multipliction rule reflection}
		R^i(\u)R^j(\v)=R^{i+j}(\u\v)
	\end{equation}
	for every $\u,\v\in\Lambda$. The action of the shift operator on an element of $\Lambda$ of the form $\mathscr{L}(\phi)$ is more involved, and is described by the following formula (see \cite[Lemma 3.9]{MonTexII}). For any $\phi\in\Q^{\mathscr{I}}$ we have
	\begin{equation}\label{eq:scrittura esplicita shifting operator}
		S\left(\mathscr{L}(\phi)\right)=\sum_{\substack{i \ \text{even}\\ 0\leq i <p}}\lambda_i\mathscr{L}(T_{p|i}\phi)+\sum_{\substack{i \ \text{odd}\\ 0\leq i <p}}\lambda_i\mathscr{L}(\overline{T_{p|i}\phi}).
	\end{equation}
	We conclude by recalling the commuting rule \cite[(10)]{MonTexII} between the operators $R$ and $S$. If $\u\in\Lambda$, then
	\begin{equation}\label{eq:rule between R and S}
		S(R(\u))=\begin{cases}\lambda_1 S(\u) &\text{ if } p=2\\ \lambda_{p-1} R(S(\u)) &\text{ if } p>2. \end{cases}
	\end{equation}

	\begin{example}\label{ex:properties of Delta}
		Let $x\in\KK\llbracket x\rrbracket$ with corresponding function $\phi_{x,p}\in \mathbb{Q}^\mathscr{I}$, so that $\phi_{x,p}(t)=t$. We denote by $\Delta=\mathscr{L}(\phi_{x,p})\in\Lambda$ the associated sequence. Since $\phi_{x,p}(t)=t$, we sometimes call $\mathscr{L}(\phi_{x,p})=\mathscr{L}(t)$. An explicit computation shows that $\Delta=(\delta_1,p^{-1}\delta_p,p^{-2}\delta_{p^2},\dots)$. 
		
		Since $\mathscr{L}(1-t)=-\mathscr{L}(t)$, we obtain that $R(\Delta)=-\Delta$. If there is a $\phi\in\Q^{\mathscr{I}}$ such that $\u=\mathscr{L}(\phi)=(u_0,u_1,u_2,\dots)$, then $\u\cdot\Delta=\alpha(u_0)\Delta$. This can be seen combining the definition of $\mathscr{L}$ in \eqref{eq:definizione serie L} with the fact that $(-1)^i\lambda_i\delta_{p^e}=\delta_{p^e}$ for any $i<p^e$ (easy consequence of the first equation in \eqref{eq:formule moltiplicazione lambda}). In particular, $\Delta\cdot\Delta=\Delta$. 
		
	\end{example}

	\section{Limit $\phi$-functions for diagonal hypersurfaces}\label{section:limitphifunctions}
	
	We consider the family of functions $\{\phi_{f,p}\}_p$ and $\{\psi_{f,p}\}_p$ for a diagonal hypersurface $f$ in prime characteristic $p$ and we investigate the behaviour of these functions as $p$ tends to infinity. We fix the following notation throughout the section.
	
	\begin{notation}\label{notation:sectionlimitphi}
		Let $n\geq2$ and $2\leq d_1\leq \cdots\leq d_n$ be integers. Let $f=x_1^{d_1}+\cdots+x_n^{d_n}$ be a diagonal hypersurface. For any prime number $p$, we consider $f$ as a hypersurface in the power series ring $A_p=\mathbb{F}_p\llbracket x_1,\dots,x_n\rrbracket$ and the corresponding functions $\phi_{f,p}$ and $\psi_{f,p}$ as introduced in \eqref{eq:defphi_p,f} and \eqref{eq:psi}.
		We fix $d=d_1\cdots d_n$ and for each prime $p$ we write $p=\ell d+r$ with integers $\ell$ and $r$ with  $0\leq r< d$. We also set $e_i=d/d_i$.
	\end{notation}

	\subsection{The limit $\displaystyle\lim_{p\to\infty}\phi_{f,p}$}
	
	We will use repeatedly the following technical result from  \cite[Theorem 2.20]{HanMon}. We report it here for the reader's convenience.
	
	\begin{theorem}[Han--Monsky]
		\label{thm:hanmonsky0}
		For $k_1, \cdots, k_n$ positive integers with each $k_i \leq p$ and $-n + \Sigma k_i$ even, take $\gamma = \frac{k_1 + \cdots + k_n - n}{2}$ and write $\prod_i \frac{(1-x^{k_i})}{1-x} = \sum_i c_i x^i$. Then $$D_{\mathbb{F}_p}(k_1,...,k_n) = \sum\limits_{\lambda \in \mathbb{Z}} c_{\gamma - p\lambda}. $$
	\end{theorem}

	\begin{theorem}\label{thm_15}
		Let $2\leq d_1\leq \cdots\leq d_n$ be integers and let $f=x_1^{d_1}+\cdots+x_n^{d_n}$.
		Then, the functions $\phi_{f,p}$ converge uniformly, as $p$ goes to infinity, to a piecewise polynomial function $\phi_f$ given by
		\[
		\phi_f(t)=\frac{d_1\cdots d_n}{2^n\cdot n!}\left(C_0(t)+2\sum_{\lambda\in\mathbb{Z}_{\geq1}}C_{\lambda}(t)\right),
		\]
		where $C_{\lambda}(t)$ for $\lambda\in\mathbb{Z}_{\geq0}$ is the piecewise polynomial function
		\[
		C_{\lambda}(t)=\sum(\epsilon_0\cdots\epsilon_n)\left(\epsilon_0t+\frac{\epsilon_1}{d_1}+\cdots+\frac{\epsilon_n}{d_n}-2\lambda\right)^n
		\]
		with the sum taken over all choices of $\epsilon_0,\dots,\epsilon_n\in\{\pm1\}$ with $\epsilon_0t+\frac{\epsilon_1}{d_1}+\cdots+\frac{\epsilon_n}{d_n}-2\lambda\geq0$. 
	\end{theorem}
	
	\begin{proof}
		We follow Notation~\ref{notation:sectionlimitphi}, so we have $d=d_1\cdots d_n$ and for each prime $p$ we write $p=\ell d+r$ with $0\leq r< d$.
		Moreover, to ease the notations, we will denote the functions $\phi_{f,p}$ and $\phi_{f}$  by $\phi_{p}$ and $\phi$  respectively, dropping the dependence on $f$.
		
		We proceed by a series of reductions that will allow us to directly apply the formula of Theorem~\ref{thm:hanmonsky0}. We first recall that on the dense set $\{\frac{a}{p^e} \}$, the function $\phi_p(t)$ is  
		$$\phi_{p}\Big(\frac{a}{p^e}\Big) = \frac{1}{p^{en}} \dim_{\mathbb{F}_p} ( A_p/(\mathfrak{m}^{[p^e]},f^a))$$
		for every $a,e\in\mathbb{Z}_{\ge 0}$ with $a\le p^e$. 
		
		First, we show that the convergence in the statement holds pointwise. 
		
		Fix $t \in [0,1]$. For each prime $p$, choose $a_p \in \mathbb{Z}$ with $0 \leq a_p \leq p$ such that $\frac{a_p}{p} \rightarrow t$ as $p \rightarrow \infty$. We will be more precise on the choice of $a_p$ later on, choosing $a_p$ to be either $\lceil pt \rceil$ or $\lfloor pt \rfloor$.  Then for each $p$, we have that
		$$|\phi_p(t) - \phi(t)| \leq \left|\phi_p(t) - \phi_p\left(\frac{a_p}{p}\right)\right| + \left|\phi_p\left(\frac{a_p}{p}\right) - \phi(t)\right| \leq \frac{C}{p} + \left|\phi_p\left(\frac{a_p}{p}\right) - \phi(t)\right| $$ where the last inequality comes from the fact that $\phi_p$ is Lipschitz (with a constant $C$ that can be taken independent of $p$ thanks to \cite[Corollary 3.3]{BST13}) and, thanks to our choice of the $a_p$, we have that $|t-\frac{a_p}{p}| \leq \frac{1}{p}$. So, it is enough to understand the convergence of the $\phi_p(\frac{a_p}{p})$ to $\phi(t)$.
		
		We now make another reduction. For any integer $0 \leq a \leq p$, consider the function $\tilde{\phi}_p(\frac{a}{p}) := \frac{1}{p^{n}} \dim_{\mathbb{F}_p} ( A_p/(x_1^{d\ell},\cdots,x_n^{d\ell},f^{a}))$. The short exact sequence
		$$ 0 \rightarrow \frac{(x_1^{d\ell}, \cdots, x_n^{d\ell},f^{a})}{(x_1^{p}, \cdots, x_n^{p},f^{a})} \rightarrow  \frac{A_p}{(\mathfrak{m}^{[p]},f^{a})} \rightarrow  \frac{A_p}{(x_1^{d\ell}, \cdots, x_n^{d\ell},f^{a})} \rightarrow 0 $$
		yields the inequality
		$$\left|\phi_p\left(\frac{a}{p}\right) - \tilde{\phi}_{p}\left(\frac{a}{p}\right)\right| \leq \frac{1}{p^n} \dim_{\mathbb{F}_p} \left(\frac{(x_1^{d\ell}, \cdots, x_n^{d\ell},f^{a})}{(x_1^{p}, \cdots, x_n^{p},f^{a})}\right) \leq  \frac{1}{p^n}\dim_{\mathbb{F}_p} \left(\frac{(x_1^{d\ell}, \cdots, x_n^{d\ell})}{(x_1^{p}, \cdots, x_n^{p})}\right)\leq \frac{C}{p} $$
		for some constant $C>0$. This last inequality follows from the fact that the quotient $(x_1^{d\ell}, \cdots, x_n^{d\ell})/(x_1^{p}, \cdots, x_n^{p})$ has an explicit basis of monomials of the form $x_1^{\beta_1}\cdots x_n^{\beta_n}$, where for each $i$ we have that $\beta_i<p$, and at least one of the $\beta_i\geq d\ell$. Since there are at least $p^n-(d\ell)^n= p^n - (p-r)^n =O(p^{n-1})$ such monomials, dividing by $p^n$ gives the desired inequality.
		
		Thus, we have reduced to showing the convergence of $\tilde{\phi}_p(\frac{a_p}{p})$ to $\phi(t)$.
		Now, let $M$ be the $\mathbb{F}_p$-object $A_p/(x_1^{d\ell},...,x_n^{d\ell})$ with $T$ acting as multiplication by $f$. Then
		\begin{equation}\label{equ:first computation phi tilde}
			\tilde{\phi}_p\left(\frac{a_p}{p}\right) = \frac{1}{p^{n}} \dim_{\mathbb{F}_p} ( A_p /(x_1^{d\ell},\cdots,x_n^{d\ell},f^{a_{p}})) = \frac{1}{p^n} \alpha_{a_p}(M).
		\end{equation}
		
		Recall that $e_i = \frac{d}{d_i}$ for each $1\leq i \leq n$. Then, $M$ is the $\mathbb{F}_p$-object $(d_1\delta_{\frac{d\ell}{d_1}})\cdots
		(d_n\delta_{\frac{d\ell}{d_n}}) \allowbreak
		=d\delta_{e_1\ell}\cdots\delta_{e_n\ell}$, and so we compute that
		$$\alpha_{a_p}(M) = d\cdot \dim_{\mathbb{F}_p}\left(\frac{\mathbb{F}_p[x_1,...,x_n]}{(x_1^{e_1\ell}, \cdots,x_n^{e_n \ell}, (\Sigma x_i)^{a_p})}\right) = d\cdot\dim_{\mathbb{F}_p}\left(\frac{\mathbb{F}_p[x_1,...,x_n,y]}{(x_1^{e_1\ell}, \cdots,x_n^{e_n \ell}, y^{a_p},y + \Sigma x_i)}\right).$$ But this is $dD_{\mathbb{F}_p}(a_p, e_1\ell, \cdots, e_n\ell)$. So now we have that 
		$$\tilde{\phi}\Big(\frac{a_p}{p}\Big) = \frac{d}{p^n}D_{\mathbb{F}_p}(a_p, e_1\ell, \cdots, e_n\ell).$$ 
		We are now in a position to apply Theorem~\ref{thm:hanmonsky0}. If we write $$(1-x)^{-n-1}(1-x^{a_p})\prod_i(1-x^{e_i\ell}) = \sum_i c_ix^i,$$ then $D_{\mathbb{F}_p}(a_p, e_1\ell, \cdots, e_n\ell) = \sum\limits_{\lambda \in \mathbb{Z}} c_{\gamma - p\lambda}$, where $\gamma = \frac{-(n+1) + a_p + \Sigma e_i\ell}{2}$. This is where the choice of $a_p$ matters: we choose $a_p$ = $\lceil pt \rceil$ or $\lfloor pt \rfloor$ so that $\gamma$ is an integer (equivalently, so that the parity condition required in Theorem \ref{thm:hanmonsky0} is satisfied). Now, $(1-x)^{-n-1}(1-x^{a_p})\prod_i(1-x^{e_i\ell})$ is a symmetric polynomial of even degree with "middle" coefficient $c_\gamma$, so it is enough to understand the coefficients $c_{\gamma - \lambda p}$ for $\lambda \geq 0$, as $c_{\gamma - \lambda p} = c_{\gamma + \lambda p}$.
		By the generalized binomial theorem,  $$(1-x)^{-n-1} = \sum\limits_{k=0}^{\infty} {n + k \choose n}x^k,$$ and so
		\begin{eqnarray*}
			c_\gamma & = & {n + \gamma \choose n} - {n + \gamma - a_p \choose n} - \sum {n + \gamma - e_i\ell \choose n} + ... \\
			& = & \sum (\epsilon_0\cdots \epsilon_{n})
			{n + \frac{1}{2}(\epsilon_{0}a_p
				+ \epsilon_1 e_1\ell + \cdots + \epsilon_n e_n\ell - n -1) \choose n}, \\
		\end{eqnarray*}
		where the sum is taken over all choices of $\epsilon_i\in\{\pm1\}$. More generally,
		\begin{equation}\label{equ:explicit coefficients c gamma - lambda p}
			c_{\gamma - \lambda p} = \sum (\epsilon_0 \cdots \epsilon_{n})
			{n + \frac{1}{2}(\epsilon_{0}a_p
				+ \epsilon_1 e_1\ell + \cdots + \epsilon_n e_n\ell - n -1) -\lambda p \choose n}.
		\end{equation}

		Write $\omega_\lambda(p)= n + \frac{1}{2}(\epsilon_{0}a_p
		+ \epsilon_1 e_1\ell + \cdots + \epsilon_n e_n\ell - n -1) -\lambda p$. We now examine the behavior of $\frac{1}{p^n}{\omega_\lambda(p) \choose n}$ as $p \rightarrow \infty$ based on the sign of 
		\begin{equation}\label{equ:limit omega over p}
			\lim_{p \to\infty} \frac{\omega_\lambda(p)}{p} = \frac{1}{2}(\epsilon_{0}t +  \frac{\epsilon_1}{d_1} + \cdots +  \frac{\epsilon_n}{d_n} - 2\lambda).
		\end{equation}
		If this limit is strictly less than $0$, then for all large enough $p$, $\omega_\lambda(p)$ is negative, and so the corresponding binomial coefficient is $0$ and does not contribute. Otherwise, for $p$ large enough
		
		\begin{eqnarray}\label{equ:limit binomial omega}
			\frac{1}{p^n}{\omega_\lambda(p) \choose n}
			& = & \frac{1}{n!p^n}(\omega_\lambda(p))(\omega_\lambda(p) - 1) \cdots (\omega_\lambda(p) - n + 1) \nonumber\\
			& = & \frac{1}{n!}\left(\left(\frac{\omega_\lambda(p)}{p}\right)\left(\frac{\omega_\lambda(p)-1}{p}\right) \cdots \left(\frac{\omega_\lambda(p) -n + 1}{p}\right)\right)
		\end{eqnarray}
		which limits to $\frac{1}{2^n n!}(\epsilon_{0}t + \epsilon_1 \frac{1}{d_1} + \cdots + \epsilon_n \frac{1}{d_n} - 2\lambda)^{n}$ as $p \rightarrow \infty$. So, we see that for all $\lambda \geq 0$, $\lim_{p \to\infty} \frac{1}{p^n}c_{\gamma - \lambda p} = \frac{1}{2^n n!}C_{\lambda}(t)$, therefore
		
		\begin{eqnarray*}
			\lim_{p \to\infty} \tilde{\phi}\left(\frac{a_p}{p}\right)  =  \lim_{p \to\infty} \frac{d}{p^n}D_{\mathbb{F}_p}(a_p, e_1\ell, \cdots, e_n\ell) 
			=  \lim_{p \to\infty} \frac{d}{p^n} \sum\limits_{\lambda \in \mathbb{Z}} c_{\gamma - p\lambda} 
			= \frac{d}{2^n n!}(C_0(t) + 2(\sum\limits_{\lambda \geq 1}C_{\lambda}(t))),
		\end{eqnarray*}
		and so $\lim_{p \to\infty} \phi_{p}(t) = \phi(t)$, as desired.
		
		Finally, recalling that the functions $\phi_p$ are bounded by the constant $1$ and Lipschitz continuous with a constant that is independent of $p$, the uniform convergence of $\{\phi_p\}_p$ to $\phi$ follows from a standard application of Arzelà-Ascoli's theorem (\textit{e.g.} see \cite[Exercise 5 on page 293]{MunkresTopology}). Otherwise, for some $\epsilon > 0$ we could exhibit a subsequence $\{\phi_{p_i}\}_i$ and points $x_i \in [0,1]$ so that  $| \phi_{p_i}(x_i) - \phi(x_i) | > \epsilon$ and $p_{i+1} > p_i$ for $i \in \mathbb{Z}_{>0}$ having no uniformly convergent subsubsequence.
	\end{proof}
	
	\begin{remark}\label{rk:points where phi and psi change polynomial}
		It is useful to remark that, since $\epsilon_0t+\frac{\epsilon_1}{d_1}+\cdots+\frac{\epsilon_n}{d_n}-2\lambda$ changes signs precisely at $t=-\epsilon_0(\frac{\epsilon_1}{d_1}+\cdots+\frac{\epsilon_n}{d_n}-2\lambda)$, the points where the functions $\phi_f$ and $\psi_f$ change polynomial expression are contained in $\{\pm\frac{1}{d_1}\pm\dots\pm\frac{1}{d_n}\pm 2\lambda\}\cap[0,1]$ for $\lambda\in\mathbb{Z}$.
	\end{remark}
	
	\begin{example}\label{ex-x2y3}
		Let $f=x^2+y^3$, we compute the limit function $\phi_{f}$ with the formula of Theorem~\ref{thm_15}. We have $n=2$, $d_1=2$, and $d_2=3$.
		First, observe that since $\frac12+\frac13<1$, we have $C_{\lambda}=0$ for all $\lambda\geq1$. Thus, we obtain
		$\phi_{f}(t)=\frac{2\cdot 3}{4\cdot 2!}C_0(t)=\frac{3}{4}C_0(t)$, and we are left to compute $C_0(t)$. From Remark~\ref{rk:points where phi and psi change polynomial}, we know that $C_0(t)$ changes polynomial expression precisely at the points $\{\frac16,\frac56\}$. For $0\leq t\leq\frac{1}{6}$, we have that
		\[
		\begin{split}
			C_{0}(t)&=\sum(\epsilon_0\epsilon_1\epsilon_2)\left(\epsilon_0t+\frac{\epsilon_1}{2}+\frac{\epsilon_2}{3}\right)^2  \\
			&=\left(t+\frac{1}{2}+\frac{1}{3}\right)^2-\left(-t+\frac{1}{2}+\frac{1}{3}\right)^2+\left(-t+\frac{1}{2}-\frac{1}{3}\right)^2-\left(t+\frac{1}{2}-\frac{1}{3}\right)^2
			=\frac{8}{3}t.    
		\end{split}
		\]
		Thus, $\phi_{f}(t)=2t$ for $0\leq t\leq\frac{1}{6}$. Similarly, one can compute the remaining two cases.    
		We obtain that the limit function of $f=x^2+y^3$ is given by
		\[
		\phi_{f}(t)=\begin{cases}
			2t &\text{ for } 0\leq t\leq \frac16\\
			-\frac{3}{2}t^2+\frac{5}{2}t-\frac{1}{24} &\text{ for } \frac16\leq t\leq \frac56\\
			1 &\text{ for } \frac56\leq t\leq 1.
		\end{cases}
		\]
	\end{example}

	\subsection{The limit derivatives of $\phi_{f,p}$}

	In this section, we study the limit of the left and right derivatives of the function $\phi_{f,p}$ as $p$ tends to infinity and we prove that they converge to the derivatives of $\phi_f$. 
	Before proving this, we need a series of preparatory lemmas. These are necessary to show that the analogous method of replacing $\phi_{f,p}(t)$ by $\tilde{\phi}_{f,p}(t)$ that we used in the proof of Theorem~\ref{thm_15} works for the derivatives as well. We keep notations as in Notation~\ref{notation:sectionlimitphi}.
	
	\begin{lemma}\label{lemma:colonproperty}
		Let $I\subseteq A_p$ be an ideal and $a,b\in\mathbb{Z}_{>0}$.
		Then $(I, f^{a+b}):f^a=((I:f^a),f^b)$.
	\end{lemma}
	\begin{proof}
		First, if $x\in ((I, f^{a+b}):f^a)$ we can write $xf^a=i+yf^{a+b}$ for some $i\in I$. Rearranging, we have that $f^a(x-yf^b)\in I$, and so $x-yf^b\in (I:f^a)$. Hence, $x\in ((I:f^a),f^b)$ as required.
		
		Conversely, if $x\in ((I:f^a),f^b)$ we can write $x=y+zf^b$ with $y\in (I:f^a)$. So, $xf^a=yf^a+zf^{a+b}\in (I,f^{a+b})$. Thus, $x$ lies in $(I, f^{a+b}):f^a$, and we are done.
	\end{proof}
	
	\begin{lemma}\label{lemma:differenceOp-2}
		For any integer $2\leq a \leq p$, we have that 
		\[
		\dim_{\mathbb{F}_p}\left(\frac{A_p}{(x_1^p,\dots,x_n^p,f^a):f^{a-2}}\right)-\dim_{\mathbb{F}_p}\left(\frac{A_p}{(x_1^{d\ell},\dots,x_n^{d\ell},f^a):f^{a-2}}\right)=O(p^{n-2})
		\]
	\end{lemma}
	
	\begin{proof}
		By applying Lemma~\ref{lemma:colonproperty}, the difference in the left hand side is equal to
		\[
		\dim_{\mathbb{F}_p}\left(\frac{A_p}{\left((x_1^p,\dots,x_n^p):f^{a-2},f^{2}\right)}\right)-\dim_{\mathbb{F}_p}\left(\frac{A_p}{\left((x_1^{d\ell},\dots,x_n^{d\ell}):f^{a-2},f^{2}\right)}\right),
		\]
		which in turn is equal to 
		\[
		\dim_{\mathbb{F}_p}\frac{\left((x_1^{d\ell},\dots,x_n^{d\ell}):f^{a-2},f^{2}\right)}{\left((x_1^p,\dots,x_n^p):f^{a-2},f^{2}\right)}
		\]
		via the standard short exact sequence
		\[
		0\rightarrow\frac{\left((x_1^{d\ell},\dots,x_n^{d\ell}):f^{a-2},f^{2}\right)}{\left((x_1^p,\dots,x_n^p):f^{a-2},f^{2}\right)}\rightarrow\frac{A_p}{\left((x_1^p,\dots,x_n^p):f^{a-2},f^{2}\right)}\rightarrow\frac{A_p}{\left((x_1^{d\ell},\dots,x_n^{d\ell}):f^{a-2},f^{2}\right)}\rightarrow0.
		\]
		Now, we fix $\delta=x_1^r\cdots x_n^r$ and $J=\left((x_1^p,\dots,x_n^p):f^{a-2},f^{2}\right)$.
		Observe that 
		\[\left((x_1^{d\ell},\dots,x_n^{d\ell}):f^{a-2},f^{2}\right)\subseteq J:\delta,\]
		since if $x\in (x_1^{d\ell},\dots,x_n^{d\ell}):f^{a-2}$ then $\delta xf^{a-2}$ is in $(x_1^p,\dots,x_n^p)$, so $\delta x\in J$.
		Thus, we have
		\[
		\dim_{\mathbb{F}_p}\frac{\left((x_1^{d\ell},\dots,x_n^{d\ell}):f^{a-2},f^{2}\right)}{\left((x_1^p,\dots,x_n^p):f^{a-2},f^{2}\right)}\leq\dim_{\mathbb{F}_p}\frac{J:\delta}{J}.
		\]
		From the short exact sequence
		\[
		0\rightarrow (J:\delta)/J\rightarrow A_p/J\xrightarrow{\ \cdot \delta} A_p/J\rightarrow A_p/(J,\delta)\rightarrow0
		\]
		by taking dimensions we obtain $\dim_{\mathbb{F}_p}\left((J:\delta)/J\right)=\dim_{\mathbb{F}_p}A_p/(J,\delta)$,
		which is less than or equal to
		\[
		\dim_{\mathbb{F}_p}\left(A_p/\left((x_1^p,\dots,x_n^p),f^2,\delta\right)\right).
		\]
		So, it is enough to prove that this last dimension is $O(p^{n-2})$.
		
		Let $I=(x_1^p,\dots,x_n^p,f^2)$. We have a filtration
		\begin{equation}\label{eq:filtration}
			(I,\delta)\subseteq (I,x_1^r\cdots x_n^{r-1})\subseteq\cdots\subseteq (I,x_1)\subseteq A_p,
		\end{equation}
		where at each step we systematically decrease the power of the last remaining $x_i$ in the second generator by one. Notice that there are $(r+1)n\leq (d+1)n$ inclusions in this sequence, a number which is independent on $p$.
		For any two consecutive terms in the previous sequence,
		\[\begin{split}
			\dim_{\mathbb{F}_p}\left(\frac{(I,x_1^{r}\cdots x_i^{j-1})}{(I,x_1^{r}\cdots x_i^{j})}\right)&=\dim_{\mathbb{F}_p}\left(\frac{A_p}{\left((I,x_1^r\cdots x_i^{j}):x_1^r\cdots x_i^{j-1}\right)}\right)\\ &\leq \dim_{\mathbb{F}_p}\left(A_p/(I,x_i)\right)= \dim_{\mathbb{F}_p}\left(A_p/(x_1^p,\dots,x_n^p,f^2,x_i)\right)
		\end{split}
		\]
		which we can readily bound by passing to the initial ideal with respect to an appropriate term order (\textit{cf.} \cite[Theorem 15.3]{EisenbudCommutativeAlgebra}).
		More precisely, we have that $A_p/(x_1^p,\dots, x_n^p,x_i)$ has a $\mathbb{F}_p$-vector basis given by monomials of the form 
		\begin{equation}\label{eq:basismonomials}
			\left\{x_1^{\beta_1}\cdots x_{i-1}^{\beta_{i-1}}x_{i+1}^{\beta_{i+1}}\cdots x_{n}^{\beta_{n}}\mid \ 0\leq\beta_1,\dots,\beta_n\leq p-1\right\},
		\end{equation}
		in particular its dimension is $O(p^{n-1})$.
		Now, recall that $f=x_1^{d_1}+\cdots+x_n^{d_n}$ and consider the image of $f$ in the quotient ring $A_p/(x_1^p,\dots, x_n^p,x_i)$. The largest $d_j$ remaining (which is  $d_{n-1}$ if $i=n$ and $d_n$ otherwise) appears in $f^2$ with exponent $2d_j$, and, in any other term containing $x_j$ in $f^2$, $x_j$ appears to a power which is not larger than $2d_j$. Thus, we can write a spanning set of $A_p/(x_1^p,\dots,x_n^p,f^2,x_i)$ with the monomials as in \eqref{eq:basismonomials}, but with $\beta_j\leq 2 d_j$ a constant independent of $p$.
		Thus, we have $\dim_{\mathbb{F}_p}\left(A_p/(x_1^p,\dots,x_n^p,f^2,x_i)\right)=O(p^{n-2})$.
		Therefore, the $\mathbb{F}_p$-dimension of any two consecutive terms in the filtration \eqref{eq:filtration} is $O(p^{n-2})$. Since, there are no more than $(d+1)n$ terms in the filtration, a number independent of $p$, we have that  $\dim_{\mathbb{F}_p}\left(A_p/\left((x_1^p,\dots,x_n^p),f^2,\delta\right)\right)=O(p^{n-2})$, which concludes the proof.
	\end{proof}
	
	\begin{lemma}\label{lemma:derivativesquotient}
		Let $\mu_p:[a,b]\rightarrow[0,1]$ be a sequence of decreasing convex functions indexed by primes $p$ converging to a $C^1$ function $\mu$. Furthermore, suppose that for fixed $t$ and for each $p$ there exist $a<\alpha_p<\beta_p<t<\gamma_p<\delta_p<b$ such that the quotients $\frac{\mu_p(\beta_p)-\mu_p(\alpha_p)}{\beta_p-\alpha_p}$ and $\frac{\mu_p(\delta_p)-\mu_p(\gamma_p)}{\delta_p-\gamma_p}$ both converge to $\mu'(t)$ as $p\rightarrow\infty$. Then, $\partial_{-}\mu_p(t)$ and $\partial_{+}\mu_p(t)$ converge to $\mu'(t)$ as $p\rightarrow\infty$.
	\end{lemma}
	
	\begin{proof}
		Since $\mu_p$ is decreasing and convex, for each $p$ we have
		\[
		\partial_{+}\mu_p(t)\leq \frac{\mu_p(\delta_p)-\mu_p(t)}{\delta_p-t}\leq\frac{\mu_p(\delta_p)-\mu_p(\gamma_p)}{\delta_p-\gamma_p}
		\]
		and
		\[
		\partial_{-}\mu_p(t)\geq \frac{\mu_p(t)-\mu_p( \alpha_p)}{t-\alpha_p}\geq\frac{\mu_p(\beta_p)-\mu_p(\alpha_p)}{\beta_p-\alpha_p}.
		\]
		Combining these inequalities with the fact that $\partial_{-}\mu_p(t)\leq \partial_{+}\mu_p(t)$, we get that
		\begin{equation}\label{eq:disuguaglianza derivate parziali}
			\frac{\mu_p(\beta_p)-\mu_p(\alpha_p)}{\beta_p-\alpha_p}\leq \partial_{-}\mu_p(t)\leq \partial_{+}\mu_p(t)\leq \frac{\mu_p(\delta_p)-\mu_p(\gamma_p)}{\delta_p-\gamma_p}.
		\end{equation}
		Taking the limit for $p\rightarrow\infty$ yields the desired result.
	\end{proof}
	
	We are now ready to state and prove the main result of this subsection.
	
	\begin{theorem}\label{thm16}
		Let $\phi_{f,p}$ and $\phi_f$ be as in Theorem~\ref{thm_15}. Then, for any $t\in[0,1]$, the sequence of derivatives $\partial_{-}\phi_{f,p}(t)$ converges to $\partial_{-}\phi_f(t)$ and $\partial_{+}\phi_{f,p}(t)$ converges to $\partial_{+}\phi_f(t)$. In particular, for all but finitely many explicit points, both sequences converge to $\phi_f'(t)$.
	\end{theorem}
	
	\begin{proof}
		As in Notation~\ref{notation:sectionlimitphi}, we fix $d=d_1\cdots d_n$, we write $p=\ell d+r$ with $0\leq r< d$ and set $e_i=d/d_i$. For simplicity, denote $\phi_{f,p}$ and $\phi_{f}$  by $\phi_{p}$ and $\phi$  respectively.
		
		First, we consider $t$ in the interior of one of the intervals where $\phi(t)$ is defined by a single polynomial expression. We choose $a_p$ to be   either $\lceil pt \rceil$ or $\lfloor pt \rfloor$ so that $ \frac{-(n+1) + a_p + \Sigma e_i\ell}{2}$ is an integer (with the aim of applying Theorem \ref{thm:hanmonsky0}).
		We want to apply Lemma~\ref{lemma:derivativesquotient} to the decreasing convex function $\psi_p(t)=\psi_{f,p}(t)=1-\phi_p(t)$ which limits to $\psi(t)=\psi_f(t)=1-\phi(t)$. Since $\partial_{\pm}\psi_{p}(t)=-\partial_{\pm}\phi_{p}(t)$ and $\psi'(t)=-\phi'(t)$ the result will follow.
		We fix $\alpha_p=\frac{a_p-4}{p}$, $\beta_p=\frac{a_p-2}{p}$, $\gamma_p=\frac{a_p+2}{p}$, $\delta_p=\frac{a_p+4}{p}$ and we check that $\frac{\psi_p(\beta_p)-\psi_p(\alpha_p)}{\beta_p-\alpha_p}$ converges to $\psi'(t)$ as $p\to\infty$. The convergence of the quotients  $\frac{\psi_p(\delta_p)-\psi_p(\gamma_p)}{\delta_p-\gamma_p}$  to $\psi'(t)$ can be checked similarly.
		
		As in the proof of Theorem~\ref{thm_15}, we consider the auxiliary function 
		\[\tilde{\phi}_p\left(\frac{a}{p}\right) = \frac{1}{p^{n}} \dim_{\mathbb{F}_p} ( A_p/(x_1^{d\ell},\cdots,x_n^{d\ell},f^{a}))\] 
		for any $0 \leq a \leq p$.
		We obtain the upper bound
		\begin{equation}\label{equ:prima approssimazione derivate}
			\begin{split}
				\left|\frac{\psi_p(\beta_p)-\psi_p(\alpha_p)}{\beta_p-\alpha_p}-\psi'(t)\right| &= \left|\frac{p}{2}\left(\psi_p\left(\frac{a_p-2}{p}\right)-\psi_p\left(\frac{a_p-4}{p}\right)\right) -\psi'(t)\right|
				\\ &= \left|\frac{p}{2}\left(\phi_p\left(\frac{a_p-4}{p}\right)-\phi_p\left(\frac{a_p-2}{p}\right)\right) -\psi'(t)\right|\\
				&\leq\left|\frac{p}{2}\left(\phi_p\left(\frac{a_p-4}{p}\right)-\phi_p\left(\frac{a_p-2}{p}\right) -\tilde{\phi}_p\left(\frac{a_p-4}{p}\right)+\tilde{\phi}_p\left(\frac{a_p-2}{p}\right)\right)\right|\\
				&+\left|\frac{p}{2}\left(\tilde{\phi}_p\left(\frac{a_p-4}{p}\right)-\tilde{\phi}_p\left(\frac{a_p-2}{p}\right)\right) -\psi'(t)\right|.
			\end{split}
		\end{equation}
		We argue separately that each of the two summands after the inequality goes to $0$ for $p\to\infty$.
		
		For the first term, consider the short exact sequence
		\[
		0\rightarrow \frac{A_p}{\left((x_1^p,\dots,x_n^p,f^{a_p-2}):f^{a_p-4}\right)} \rightarrow\frac{A_p}{\left(x_1^p,\dots,x_n^p,f^{a_p-2}\right)}\rightarrow\frac{A_p}{\left(x_1^p,\dots,x_n^p,f^{a_p-4}\right)}\rightarrow0
		\]
		and the analogous short exact sequence where the powers of $p$ are replaced by powers of $d\ell$.
		By taking dimensions over $\mathbb{F}_p$, we find that the first summand after the inequality is 
		\[
		\frac{1}{2p^{n-1}}\left(\dim_{\mathbb{F}_p}\left(\frac{A_p}{\left((x_1^p,\dots,x_n^p,f^{a_p-2}):f^{a_p-4}\right)}\right)-\dim_{\mathbb{F}_p}\left(\frac{A_p}{\left((x_1^{d\ell},\dots,x_n^{d\ell},f^{a_p-2}):f^{a_p-4}\right)}\right)\right).
		\]
		By Lemma~\ref{lemma:differenceOp-2}, the difference inside the parenthesis is $O(p^{n-2})$. In particular, the whole term goes to $0$ when $p\to\infty$. 
		
		Now, we are left to show that 
		\[
		\frac{p}{2}\left(\tilde{\phi}_p\left(\frac{a_p-4}{p}\right)-\tilde{\phi}_p\left(\frac{a_p-2}{p}\right)\right) \to \psi'(t)
		\]
		as $p\to\infty$. In order to compute explicitely the values of $\tilde{\phi}_p$, one can follow verbatim the proof of Theorem \ref{thm_15} from \eqref{equ:first computation phi tilde} to \eqref{equ:explicit coefficients c gamma - lambda p}, with $a_p$ replaced by $a_p-2$ and $a_p-4$. Notice that we can apply Theorem \ref{thm:hanmonsky0} exactly as in the proof of Theorem \ref{thm_15} thanks to our choice of $a_p$ and to the fact that $a_p-2$ and $a_p-4$ have the same parity of $a_p$. Then, we obtain that $\frac{p}{2}\left(\tilde{\phi}_p\left(\frac{a_p-4}{p}\right)-\tilde{\phi}_p\left(\frac{a_p-2}{p}\right)\right)$ coincides with
		\begin{equation*}
			\frac{d}{2p^{n-1}}\sum_{\lambda\in\mathbb{Z}}\sum(\epsilon_0\cdots\epsilon_n)\Big({\omega_\lambda(p)-2\epsilon_0 \choose n}- {\omega_\lambda(p)-\epsilon_0 \choose n}\Big),
		\end{equation*}
		where again $\omega_\lambda(p)=n + \frac{1}{2}(\epsilon_{0}a_p
		+ \epsilon_1 e_1\ell + \cdots + \epsilon_n e_n\ell - n -1) -\lambda p$ and the second sum is taken over all choices of $\epsilon_i\in\{\pm1\}$. The recursive formula for binomial coefficients implies that, if $\epsilon_0=1$,
		\begin{equation*}
			{\omega_\lambda(p)-2\epsilon_0 \choose n}- {\omega_\lambda(p)-\epsilon_0 \choose n}= -{\omega_\lambda(p)-2 \choose n-1},
		\end{equation*}
		whereas when $\epsilon_0=-1$,
		\begin{equation*}
			{\omega_\lambda(p)-2\epsilon_0 \choose n}- {\omega_\lambda(p)-\epsilon_0 \choose n}= {\omega_\lambda(p)+1 \choose n-1}.
		\end{equation*}
		In either case, the limit of 
		\begin{equation*}
			\frac{1}{p^{n-1}}\Big({\omega_\lambda(p)-2\epsilon_0 \choose n}- {\omega_\lambda(p)-\epsilon_0 \choose n}\Big)
		\end{equation*}
		as $p\to\infty$ can be computed exactly as in \eqref{equ:limit binomial omega}, giving
		\begin{equation*}
			\frac{-\epsilon_0}{2^{n-1}(n-1)!}\left(\epsilon_{0}t +  \frac{\epsilon_1}{d_1} + \cdots +  \frac{\epsilon_n}{d_n} - 2\lambda\right)^{n-1}
		\end{equation*}
		if $\lim_{p\to\infty}\frac{\omega_\lambda(p)}{p}$ is nonnegative, and $0$ otherwise. As noted in \eqref{equ:limit omega over p},
		\begin{equation*}
			\lim_{p \to\infty} \frac{\omega_\lambda(p)}{p} = \frac{1}{2}(\epsilon_{0}t + \frac{\epsilon_1 }{d_1} + \cdots +  \frac{\epsilon_n}{d_n} - 2\lambda),
		\end{equation*}
		therefore we obtain that $\lim_{p\to\infty}\frac{p}{2}\left(\tilde{\phi}_p\left(\frac{a_p-4}{p}\right)-\tilde{\phi}_p\left(\frac{a_p-2}{p}\right)\right)$ is equal to
		\begin{equation}\label{eq:psiderivative}
			-\frac{d}{2^n(n-1)!}\sum_{\lambda\in\mathbb{Z}}\sum(\epsilon_0^2\epsilon_1\cdots\epsilon_n)\left(\epsilon_{0}t + \frac{\epsilon_1 }{d_1} + \cdots +  \frac{\epsilon_n}{d_n} - 2\lambda\right)^{n-1},
		\end{equation}
		where the second sum is taken over all choices of $\epsilon_i\in\{\pm 1\}$ such that $\epsilon_{0}t + \frac{\epsilon_1 }{d_1} + \cdots +  \frac{\epsilon_n}{d_n} - 2\lambda\ge 0$. But this limit is precisely $\psi'(t)$, and so we are done.
		
		In order to finish the proof, we need to check the case where $\tilde{t}$ is a boundary point for one of the polynomial pieces of $\phi$ (and hence of $\psi=1-\phi$). In order to do that, we need a more precise estimate for
		\begin{equation*}
			\Big|\frac{p}{2}\left(\tilde{\phi}_p\left(\frac{a_p-4}{p}\right)-\tilde{\phi}_p\left(\frac{a_p-2}{p}\right)\right)-\psi'(t)\Big|
		\end{equation*}
		for $t$ in the interior of a polynomial piece of $\psi$. First, by an explicit computation, we can bound
		\begin{equation*}
			\Big|\frac{\epsilon_0 a_p}{p}-\epsilon_0 t\Big|\le \frac{1}{p}\quad\text{and}\quad \Big|\frac{\epsilon_i e_i \ell}{p}-\frac{\epsilon_i}{d_i}\Big|\le \frac{e_i}{p}
		\end{equation*}
		for every $1\le i\le n$. Therefore, for any constant $D$ bounded by $2n$ we have that
		\begin{equation*}
			\Big|\frac{\omega_\lambda(p)+D}{p}- \frac{1}{2}(\epsilon_{0}t +  \frac{\epsilon_1}{d_1} + \cdots +  \frac{\epsilon_n}{d_n} - 2\lambda)\Big|\le\frac{c}{p}
		\end{equation*}
		for some constant $c$ independent on $t$. The terms with $|\lambda|>n$ vanish when we sum over the signs. For the remaining values of $\lambda$, all the factors are bounded by a constant depending only on $f$, so we can take products and conclude that
		\begin{multline*}
			\Big|  \frac{1}{p^{n-1}}\Big({\omega_\lambda(p)-2\epsilon_0 \choose n}- {\omega_\lambda(p)- \epsilon_0 \choose n}\Big) \\ -\frac{-\epsilon_0}{2^{n-1}(n-1)!}
			\left(\max\left\{\epsilon_0t+\frac{\epsilon_1}{d_1}+\cdots+
			\frac{\epsilon_n}{d_n}-2\lambda,0\right\}\right)^{n-1}\Big|\le \frac{C}{p}
		\end{multline*}
		for a constant $C$ independent on $t$. Taking finite sums, we find another constant $C$ independent on $t$ such that
		\begin{equation*}
			\Big|\frac{p}{2}\left(\tilde{\phi}_p\left(\frac{a_p-4}{p}\right)-\tilde{\phi}_p\left(\frac{a_p-2}{p}\right)\right)-\psi'(t)\Big|\le\frac{C}{p}.
		\end{equation*}
		Coming back to equation \eqref{equ:prima approssimazione derivate}, we obtain that there is a constant $C$ independent on $t$ such that
		\begin{equation*}
			\left|\frac{\psi_p(\beta_p)-\psi_p(\alpha_p)}{\beta_p-\alpha_p}-\psi'(t)\right|\le \frac{C}{p}.
		\end{equation*}
		The same argument for $\gamma_p$ and $\delta_p$ in place of $\alpha_p$ and $\beta_p$ respectively yields that there is a constant $C$ independent on $t$ such that
		\begin{equation*}
			\left|\frac{\psi_p(\beta_p)-\psi_p(\alpha_p)}{\beta_p-\alpha_p}-\psi'(t)\right|+ \left|\psi'(t)-\frac{\psi_p(\delta_p)-\psi_p(\gamma_p)}{\delta_p-\gamma_p}\right|\le\frac{C}{p}.
		\end{equation*}
		Then, inequality \eqref{eq:disuguaglianza derivate parziali} implies that
		\begin{equation*}
			|\partial_-\psi_p(t)-\psi'(t)|\le \frac{C}{p} \quad\text{and}\quad |\partial_+\psi_p(t)-\psi'(t)|\le \frac{C}{p}.
		\end{equation*}
		
		Let now $\tilde{t} \in (0,1)$ be a boundary point for one of the polynomial pieces of $\phi$ (and hence of $\psi=1-\phi$). We handle the left derivative $\partial_{-}\psi(\tilde{t})$, with the right derivative being dealt analogously. Since $\psi_p$ is convex, the left derivative function $\partial_{-}\psi_p(t)$ is left continuous.
		Using the above inequality, we obtain that 
		\[
		|\partial_-\psi_p(\tilde{t})-\partial_-\psi(\tilde{t})|=\lim_{t\to\tilde{t}^-}|\partial_- \psi_p(t)-\psi'(t)|\le\frac{C}{p}.
		\]
		Letting $p\to\infty$ gives the desired result.
		
		It remains to treat the endpoints $\tilde{t}=0$ and $\tilde{t}=1$.
		We claim that, for all sufficiently large $p$,
		\begin{align*}
			\partial_+\phi_p(0)
			&=d_n\phi_{x_1^{d_1}+\cdots+x_{n-1}^{d_{n-1}},p}
			\left(\frac{1}{d_n}\right),\\
			\partial_-\phi_p(1)
			&=d_n\left(1-
			\phi_{x_1^{d_1}+\cdots+x_{n-1}^{d_{n-1}},p}
			\left(1-\frac{1}{d_n}\right)\right).
		\end{align*}
		Indeed, $A_p/f$ is free over
		$\mathbb{F}_p\llbracket x_1,\ldots,x_{n-1}\rrbracket$
		with basis $1,x_n,\ldots,x_n^{d_n-1}$.
		Moreover, $f^{p^e}=x_n^{d_np^e}$ modulo
		$(x_1^{p^e},\ldots,x_{n-1}^{p^e})$.
		As $x_n$ and hence also $f$ are nonzerodivisors modulo
		$(x_1^{p^e},\ldots,x_{n-1}^{p^e})$, it follows that
		\[
		(x_1^{p^e},\ldots,x_n^{p^e}):f^{p^e-1}
		=(x_1^{p^e},\ldots,x_{n-1}^{p^e},f):
		x_n^{(d_n-1)p^e}.
		\]
		Since the kernel and cokernel of multiplication by
		$x_n^{(d_n-1)p^e}$ on
		$A_p/(x_1^{p^e},\ldots,x_{n-1}^{p^e},f)$
		have the same length, \eqref{eq:psi} gives
		\[
		p^e\left(1-\phi_p\left(1-\frac{1}{p^e}\right)\right)
		=d_n-\frac{1}{p^{(n-1)e}}
		\dim_{\mathbb F_p}\left(
		A_p/(f,x_1^{p^e},\ldots,x_{n-1}^{p^e},
		x_n^{(d_n-1)p^e})\right).
		\]
		For every integer $0\leq k\leq p^e$, the relation
		$x_n^{d_nk}=(-1)^k(x_1^{d_1}+\cdots+x_{n-1}^{d_{n-1}})^k$
		in $A_p/f$ and the free basis above give
		\[
		\dim_{\mathbb F_p}\left(
		A_p/(f,x_1^{p^e},\ldots,x_{n-1}^{p^e},x_n^{d_nk})
		\right)
		=d_np^{(n-1)e}
		\phi_{x_1^{d_1}+\cdots+x_{n-1}^{d_{n-1}},p}
		\left(\frac{k}{p^e}\right).
		\]
		Taking $k$ to be the floor and ceiling of $p^e/d_n$
		bounds the quotient length defining $\phi_p(1/p^e)$;
		using $(d_n-1)p^e/d_n$ instead bounds the quotient
		length in the preceding formula.
		Dividing by $p^{(n-1)e}$ and letting $e\to\infty$
		gives both derivative identities by continuity.
		
		For $n\geq3$, Theorem~\ref{thm_15} gives the limits of
		the two right-hand sides as $p\to\infty$; for $n=2$, use
		$\phi_{x_1^{d_1},p}(t)=\min\{1,d_1t\}$.
		At $t=0$, the comparison with \eqref{eq:psiderivative} is termwise. 
		At $t=1$, shifting $\lambda$ by one gives a telescoping sum whose boundary term is $\frac{d}{2^{n-1}(n-1)!} P^{d_1,\ldots,d_{n-1}}_{n-1}(1-1/d_n)=d_n$, by Lemma~\ref{lemma:combinatorics}. Hence
		\begin{align*}
			\lim_{p\to\infty}\partial_+\phi_p(0)
			&=-\partial_+\psi(0)=\partial_+\phi(0),\\
			\lim_{p\to\infty}\partial_-\phi_p(1)
			&=-\partial_-\psi(1)=\partial_-\phi(1). \qedhere
		\end{align*}
	\end{proof}
	
	With out major results of the section in hand, we show how we can recover information about some invariants of the hypersurface $f$, the LCT, the Hilbert-Kunz multiplicity and the $F$-signature, from the limit functions $\phi_f$ and $\psi_f$.
	
	We recall that for a diagonal hypersurface $f=x_1^{d_1}+\cdots+x_n^{d_n}$ the log canonical threshold has value $\mathrm{LCT}(f)= \min\{\sum_{i=1}^n1/d_i,1\}$. We refer to \cite{BFS13} for the definition and basic properties of this invariant.
	
	\begin{corollary}
		Let $f=x_1^{d_1}+\cdots+x_n^{d_n}$ be as in Notation~\ref{notation:sectionlimitphi}, and assume $\mathrm{LCT}(f)= \sum_{i=1}^n1/d_i \leq1$. Then, for all sufficiently small $0<\varepsilon\ll 1$ and $t\in[\mathrm{LCT}(f)-\varepsilon,\mathrm{LCT}(f)]$, we have
		\[
		\psi_f(t)=\frac{d_1\cdots d_n}{2^{n-1}n!}\left(\mathrm{LCT}(f)-t\right)^n.    
		\]
	\end{corollary}
	
	\begin{proof}
		We use the formula for $\psi_f(t)=1-\phi_f(t)$ given in Theorem~\ref{thm_15}. First, observe that since $\sum_{i=1}^n1/d_i\leq 1$, the $C_{\lambda}(t)$ are identically $0$ on $[0,1]$ for all $\lambda\geq1$. So, $$\psi_f(t)=1-\phi_f(t)=1-\frac{d_1\cdots d_n}{2^nn!}C_0(t).$$ 
		Moreover, notice that for $t\in [\mathrm{LCT}(f),1]$, we have
		\begin{equation*}
			C_0(t)=\sum_{\epsilon_0=1}(1\cdot\epsilon_1\cdots\epsilon_n)\left(1\cdot t+\frac{\epsilon_1}{d_1}+\cdots+\frac{\epsilon_n}{d_n}\right)^n.
		\end{equation*}
		So, the function $\psi_f|_{[\mathrm{LCT}(f),1]}$ consists of a unique polynomial in the variable $t$. We show that this polynomial is identically $0$.
		Since $\psi_{f,p}$ vanishes above the $F$-pure threshold $\mathrm{fpt}(A_p/f)$ for each $p$ (see \cite{BST13}) and $\mathrm{LCT}(f)\geq\mathrm{fpt}(A_p/f)$ for all sufficiently large $p$, we must have that the limit function $\psi_f$ is identically $0$ on the interval $[\mathrm{LCT}(f),1]$. If $\mathrm{LCT}(f)<1$, then the interval $[\mathrm{LCT}(f),1]$ has non-empty interior, and we are done. If $\mathrm{LCT}(f)=1$, then by Lemma~\ref{lemma:combinatorics} we have that $C_0(t)=P_n^{d_1,\dots,d_n}(t)=\frac{2^n\cdot n!}{d_1\cdots d_n}$. So, also in this case $\psi_f|_{[\mathrm{LCT}(f),1]}=0$.
		
		We next note that the only difference between the polynomial description of $C_0(t)$ on $[\mathrm{LCT}(f),1]$ and $[\mathrm{LCT}(f)-\varepsilon,\mathrm{LCT}(f)]$, for $\varepsilon$ sufficiently small, is that the term $(-1)^n(t-\mathrm{LCT}(f))^n$ becomes $-(-t+\mathrm{LCT}(f))^n$. So, we can compute
		\begin{equation*}
			\begin{split}
				\left.\psi_f\right|_{[\mathrm{LCT}(f)-\varepsilon,\mathrm{LCT}(f)]}(t)&=\left.\psi_f\right|_{[\mathrm{LCT}(f)-\varepsilon,\mathrm{LCT}(f)]}(t)-\left.\psi_f\right|_{[\mathrm{LCT}(f),1]}(t)\\
				&=\frac{d_1\cdots d_n}{2^nn!}\left(\left.(C_0)\right|_{[\mathrm{LCT}(f),1]}(t)-\left.(C_0)\right|_{[\mathrm{LCT}(f)-\varepsilon,\mathrm{LCT}(f)]}(t)\right)\\
				&=\frac{d_1\cdots d_n}{2^nn!}\left((-1)^n(t-\mathrm{LCT}(f))^n+(-t+\mathrm{LCT}(f))^n\right)\\
				&=\frac{d_1\cdots d_n}{2^{n-1}n!}(\mathrm{LCT}(f)-t)^n
			\end{split}
		\end{equation*}
		as elements of $\mathbb{Q}[t]$, hence we are done.
	\end{proof}
	
	\begin{lemma}\label{lemma:combinatorics}
		For each set of integers $n\in\mathbb{Z}_{>0}$, $\ell \in \mathbb{Z}_{\geq 0}$, and $d_1,\ldots, d_n \in \mathbb{Z}_{>0}$, set
		\[
		P^{d_1, \ldots, d_n}_\ell(t) := \sum_{\epsilon_i \in \{\pm 1\}} (\epsilon_1 \cdots \epsilon_n)\left(t + \frac{\epsilon_1}{d_1} + \cdots + \frac{\epsilon_n}{d_n}\right)^\ell
		\in \mathbb{Q}[t].\]
		Then,  we have
		\begin{equation}
			\label{combinatorialclaim}
			P^{d_1, \ldots, d_n}_{n}(t)
			= \frac{2^n \cdot n!}{d_1\cdots d_n}.
		\end{equation}
	\end{lemma}
	
	\begin{proof}
		First, $P^{d_1,\ldots,d_n}_{n+1}(0)=0$, since the terms cancel in pairs when all signs are reversed. Thus, for any $t>0$ we have the following identity
		\[
		\int_0^t P^{d_1, \ldots, d_n}_{n}(s) \, ds = \frac{1}{n + 1} P^{d_1, \ldots, d_n}_{n+ 1}(t).
		\]
		We prove by induction on $n$ that the formula \eqref{combinatorialclaim} holds.
		For $n = 1$, we have
		\[
		P^{d_1}_1 (t)= \left(t + \frac{1}{d_1}\right)^1 - \left(t - \frac{1}{d_1}\right)^1 = \frac{2}{d_1}
		\]
		as desired. Assume now that \eqref{combinatorialclaim} holds for some $n$ and we have $d_1, \ldots, d_{n+1} \in \mathbb{Z}_{>0}$. Integrating as above, we get that
		\[
		\frac{1}{n+1}P^{d_1, \ldots, d_n}_{n+1}(t)
		=\int_0^t P^{d_1, \ldots, d_n}_{n}(s) \, ds  =\frac{2^n \cdot n!}{d_1\cdots d_n}t, \quad \mbox{ that is }\quad P^{d_1, \ldots, d_n}_{n+1}(t) = \frac{2^n \cdot (n+1)!}{d_1\cdots d_n}t.
		\]
		Thus, we may compute
		\begin{align*}
			P^{d_1, \ldots, d_{n+1}}_{n+1}(t) & = P^{d_1, \ldots, d_n}_{n+1}\left(t+ \frac{1}{d_{n+1}}\right) - P^{d_1, \ldots, d_n}_{n+1}\left(t- \frac{1}{d_{n+1}}\right) \\
			& = \frac{2^n \cdot (n+1)!}{d_1\cdots d_n}\left(t+ \frac{1}{d_{n+1}}\right) - \frac{2^n \cdot (n+1)!}{d_1\cdots d_n}\left(t- \frac{1}{d_{n+1}}\right)\\
			&= \frac{2^n \cdot (n+1)!}{d_1\cdots d_n} \left(t +  \frac{1}{d_{n+1}} - t + \frac{1}{d_{n+1}}\right)\\
			&= \frac{2^n \cdot (n+1)!}{d_1\cdots d_n} \left(\frac{2}{d_{n+1}}\right) = \frac{2^{n+1} \cdot (n+1)!}{d_1\cdots d_{n+1}}
		\end{align*}
		as desired, completing the induction.
	\end{proof}

	We can give a new proof of the following result, which was obtained by Gessel and Monsky using a different strategy \cite{GesselMonsky}.
	
	\begin{corollary}[Gessel--Monsky]\label{cor:limitHK}
		Let $A_p=\mathbb{F}_p\llbracket x_1,\dots,x_n\rrbracket$ and $f=x_1^{d_1}+\cdots+x_n^{d_n}\in A_p$ be as in Notation~\ref{notation:sectionlimitphi}. Then,
		\[
		\lim_{p\to\infty}e_{\HK}(A_p/f)=\frac{d_1\cdots d_n}{2^{n-1}(n-1)!}\left(D_0+2\sum_{\lambda\in\mathbb{Z}_{\geq1}}D_\lambda\right),     
		\]
		where $D_\lambda$ for $\lambda\in\mathbb{Z}_{\geq0}$ is
		\[
		D_{\lambda}=\sum(\epsilon_1\cdots\epsilon_n)\left(\frac{\epsilon_1}{d_1}+\cdots+\frac{\epsilon_n}{d_n}-2\lambda\right)^{n-1}
		\]
		with the sum taken over all choices of $\epsilon_1,\dots,\epsilon_n\in\{\pm1\}$ with $\epsilon_1\frac{1}{d_1}+\cdots+\epsilon_n\frac{1}{d_n}-2\lambda\geq0$. 
	\end{corollary}
	
	\begin{proof}
		We observe that 
		\[
		\partial_{+}C_{\lambda}(0)=n\sum\varepsilon_0^2(\varepsilon_1\cdots\varepsilon_n)\left(\varepsilon_0\cdot 0+\frac{\epsilon_1}{d_1}+\cdots+\frac{\epsilon_n}{d_n}-2\lambda\right)^{n-1}=2n D_\lambda,
		\]
		where the $2$ comes from the fact that each term appears twice (for $\varepsilon_0=1$ and $\varepsilon_0=-1$). So, as noticed in Remark \ref{rk:analytic properties phi and psi}, applying Theorem \ref{thm16} we can compute
		\[
		\lim_{p\to\infty}e_{\HK}(A_p/f)=-\lim_{p\to\infty}\partial_{+}\psi_{f,p}(0)=-\psi_f'(0)=\phi_f'(0)=\frac{d_1\cdots d_n}{2^nn!}\left(C_0'(0)+2\sum_{\lambda\in\mathbb{Z}_{\geq1}}C'_\lambda(0)\right),
		\]
		which is precisely $\displaystyle\frac{d_1\cdots d_n}{2^{n-1}(n-1)!}\left(D_0+2\sum_{\lambda\in\mathbb{Z}_{\geq1}}D_\lambda\right)$ as desired.
	\end{proof}
	
	We give an analogous statement for the limit $F$-signature. We point out that in his senior thesis, the second author has also obtained a formula for the limit $F$-signature of diagonal hypersurfaces \cite{Shideler13}. However, the formula we get here appears (on its face) to be different.
	
	\begin{corollary}\label{cor:Fsignaturediagonal}
		Let $A_p=\mathbb{F}_p\llbracket x_1,\dots,x_n\rrbracket$ and $f=x_1^{d_1}+\cdots+x_n^{d_n}\in A_p$ be as in Notation~\ref{notation:sectionlimitphi}. Then,
		\[
		\lim_{p\to\infty}\fs(A_p/f)=\frac{d_1\cdots d_n}{2^{n}(n-1)!}\left(B_0+2\sum_{\lambda\in\mathbb{Z}_{\geq1}}B_\lambda\right),     
		\]
		where $B_\lambda$ for $\lambda\in\mathbb{Z}_{\geq0}$ is
		\[
		B_{\lambda}=\sum(\epsilon_1\cdots\epsilon_n)\left(\epsilon_0+\frac{\epsilon_1}{d_1}+\cdots+\frac{\epsilon_n}{d_n}-2\lambda\right)^{n-1}
		\]
		with the sum taken over all choices of $\epsilon_0,\epsilon_1,\dots,\epsilon_n\in\{\pm1\}$ with $\epsilon_0+\frac{\epsilon_1}{d_1}+\cdots+\frac{\epsilon_n}{d_n}-2\lambda\geq0$. 
	\end{corollary}
	
	\begin{proof}
		It follows directly by plugging $1$ into the expression for the derivative of $\phi_f(t)$ obtained from Theorem~\ref{thm_15}. See \eqref{eq:psiderivative} for the explicit formula of the derivative. 
	\end{proof}

	\section{Limit $\phi$-function of the Fermat quadric}\label{section:Fermatquadric}
	In this section, we give a combinatorial expression of the limit function $\phi_f$ of Theorem~\ref{thm_15} for the Fermat quadric $f=x_1^2+\cdots+x_n^2$ in terms of Bernoulli and Euler numbers. We also recover a formula for the limit Hilbert-Kunz multiplicity obtained by Gessel and Monsky \cite{GesselMonsky}.
	We set the following notation for this section.
	
	\begin{notation}\label{notation:sectionFermatquadric}
		Let $n\geq2$ and let $f=x_1^{2}+\cdots+x_n^{2}$ be the Fermat quadric. For any prime number $p$, we consider $f$ as a hypersurface in the power series ring $A_p=\mathbb{F}_p\llbracket x_1,\dots,x_n\rrbracket$.
	\end{notation}
	
	Bernoulli and Euler numbers are defined via Taylor expansions of some huperbolic functions. In particular, they satisfy the following classical identities
	\begin{itemize}
		\item $\mathrm{sec}(x) = \sum\limits_{k=0}^{\infty} \frac{(-1)^kE_{2k}}{(2k)!}x^{2k},$ where $E_{2k}$ is the $(2k)^{th}$ Euler number.
		
		\item $\tan(x) = \sum\limits_{k=0}^{\infty} \frac{(-1)^k 2^{2k+2}(2^{2k+2}-1) B_{2k+2}}{(2k+2)!} x^{2k+1},$ where $B_{2k+2}$ is the $(2k+2)^{th}$ Bernoulli number.
	\end{itemize}
	It turns out that the correct generalization of Euler numbers for our purposes are the so-called Euler polynomials $E_k(t)$, which are defined in terms of a power series expansion
	\begin{equation}\label{eq:Eulerpolynomial}
		\frac{2e^{xt}}{e^x + 1} = \sum_{k=0}^{\infty}E_k(t)\frac{x^k}{k!}.
	\end{equation}
	First, we collect some relevant facts about these polynomials in a lemma (see the site https://dlmf.nist.gov/24 for a database).
	
	\begin{lemma}\label{lem:eulerfacts}
		The following facts hold.
		\begin{itemize}
			\item $E_k(t)$ is a degree $k$ polynomial with leading coefficient $1$.
			\item $E_k(t+1/2)+ E_k(t-1/2) = 2(t-1/2)^k$.
			\item $\frac{d}{dt}E_k(t) = kE_{k-1}(t)$.
			\item $E_k(0) = -\frac{2}{k+1}(2^{k+1} - 1)B_{k+1}$.
			\item $E_k(1) = - E_k(0)$ for every $k\ge 1$.
			\item $E_k(\frac{1}{2}) = 2^{-k}E_k$.
		\end{itemize}
	\end{lemma}
	
	Before stating the main result of this section, we need a combinatorial lemma.
	Given a polynomial $g(t)$, we denote by $\Delta[g](t)$ the polynomial defined by $\Delta[g](t)=g(t)-g(t+1)$. Similarly, for any $m>1$ we define $\Delta^m[g](t)=\Delta[\Delta^{m-1}[g]](t)$.
	We warn the reader that this shall not be confused with the $\Delta$ of Example~\ref{ex:properties of Delta}.
	The following lemma can be proved by a straightforward induction.
	
	\begin{lemma}\label{lemma25}
		Let $m,n\in\mathbb{N}$ and $\ell\in\mathbb{Z}$. Then, the following facts hold.
		\begin{enumerate}
			\item $\displaystyle \Delta^m[t^n](t) = \sum_{j=0}^{m} (-1)^j{m \choose j}(t+j)^n$.
			\item $\displaystyle \sum_{j=0}^m(-1)^j{m \choose j}(\ell + j)^n =
			\left\{
			\begin{array}{c@{\qquad}c}
				0 & n<m \\[4pt]
				
				(-1)^nn! & n = m
			\end{array}
			\right.$.
		\end{enumerate}
	\end{lemma}
	
	\begin{theorem}\label{thm:squareslim}
		Let $f = x_1^2 + \cdots +x_n^2\in A_p=\mathbb{F}_p\llbracket x_1,\dots,x_n\rrbracket$ and let  $E_n(t)$ be the $n^{th}$ Euler polynomial as in \eqref{eq:Eulerpolynomial}. Then,
		
		\begin{enumerate}
			\item If $n$ is even, the limit function of $f$ on $[0,1]$ is given by 
			$$\phi_f(t) =  t + \frac{(-1)^{\frac{n}{2}}2^{n-1}}{n!}E_n(t).$$ 
			
			\item If $n$ is odd, the limit function  of $f$ on $[0,\frac{1}{2}]$ is given by 
			$$\phi_f(t) =  t + \frac{(-1)^{(n-1)/2}2^{n-1}}{n!}E_n\left(t+ \frac{1}{2}\right).$$
			
			\item If $n$ is odd, the limit function of $f$  on $[\frac{1}{2},1]$ is given by 
			$$ \phi_f(t)=   t - \frac{(-1)^{(n-1)/2}2^{n-1}}{n!}E_n\left(t-\frac{1}{2}\right).$$
		\end{enumerate}
	\end{theorem}
	
	\begin{proof}
		
		First, we outline the strategy. In each equality appearing in the theorem, both the left-hand and right-hand sides are polynomials of degree $n$. So, it is enough to show that these polynomials or their derivatives agree at some collection of $n+1$ points.
		
		(1) We begin with the case where $n$ is even. Because $n$ is even, all expressions of the form $ \frac{\epsilon_{1}}{2} + \cdots + \frac{\epsilon_{n}}{2}-2\lambda$ are integers. Thus, by Remark \ref{rk:points where phi and psi change polynomial}, $\phi_f(t)$ is indeed given by a single polynomial on $[0,1]$. We will show that all of the even-order derivatives of this polynomial and of $ t + \frac{(-1)^{n/2}\cdot 2^{n-1}}{n!}E_n(t)$ agree at both $t=0$ and $t=1$. This gives the requisite $n+1$ conditions, since the $n^{\mathrm{th}}$ derivative is a constant and so we only get one condition from it.
		
		From Theorem~\ref{thm_15}, the explicit expression for $\phi_f(t)$ is given by $$\mu(t) :=\frac{1}{n!}\sum_{m=n/2}^n (-1)^{(n-m)}\cdot {n \choose m}\cdot \left( \begin{array}{l} \sum_{i=0}^{\lfloor (m-n/2)/2\rfloor}(t + m - n/2 - 2\cdot i)^n \\[5pt]
			+ \sum_{i=1}^{\lfloor (m-n/2)/2\rfloor} (t + m - n/2 - 2 i)^n \\[5pt]
			- \sum_{i=0}^{\lfloor (m-n/2-1)/2\rfloor} (-t + m - n/2 - 2 i)^n \\[5pt]
			-\sum_{i=1}^{\lfloor (m-n/2-1)/2\rfloor} (-t + m - n/2 - 2 i)^n
		\end{array} \right).$$
		The ${n \choose m}$ factor comes from the fact that this is the number of ways to choose $\epsilon_{1},\dots,\epsilon_{n}$ so that $\frac{\epsilon_1}{2}+\cdots+\frac{\epsilon_n}{2} = m-\frac{n}{2}$. The first two interior sums correspond to the terms of $\phi_f(t)$ where $\epsilon_0 = 1$ and the second two interior sums correspond to the terms where $\epsilon_0 = -1$. The ranges on these sums are precisely the $i$ for which the interior expression is nonnegative for all $t \in [0,1]$.

		We need to see that this expression equals $$\nu(t) := t + \frac{(-1)^{n/2}\cdot 2^{(n-1)}}{n!} \cdot E_n(t).$$
		We first check that the leading terms of the two polynomials agree. This is equivalent to showing that $\mu^{(n)}(t)$ = $\nu^{(n)}(t)$. Since the leading coefficient of any Euler polynomial is 1, the leading coefficient of $\nu(t)$ is $\frac{(-1)^{n/2}\cdot 2^{(n-1)}}{n!}$.
		
		In the expression for $\mu(t)$, whenever $(m-\frac{n}{2} - 2 i)$ is nonzero, the $t^n$ term of $(t+m-n/2-2 i)^n$ is cancelled by the $t^n$ term of $-(-t + m - n/2 - 2 i)^n$. Thus the only terms which contribute to the leading coefficient of $\mu(t)$ are those which occur when $(m- n/2 - 2 i) = 0$,  that is, when $i = \frac{1}{2}(m- \frac{n}{2})$. In particular, we need $m-\frac{n}{2}$ even. When $m = n/2$, we get one such term from the first sum in the expression for $\mu(t)$ with $i=0$. Otherwise, such a term appears twice, once in the first sum and once in the second sum in the expression for $\mu(t)$. We now distinguish between the cases when $\frac{n}{2}$ is even or odd. If $\frac{n}{2}$ is even, the leading term of $\mu(t)$ is given by $$\frac{1}{n!} \left({n \choose n/2} + 2\sum_{j = n/4 +1}^{n/2} {n \choose 2j}  \right) = \frac{1}{n!}\left(\sum_{j=0}^{n/2}{n \choose 2j} \right) = \frac{2^{n-1}}{n!}. $$
		If $\frac{n}{2}$ is odd, the leading term is given by $$\frac{-1}{n!}\left( {n \choose n/2} + 2\sum_{j = (n+2)/4}^{(n-2)/2} {n \choose 2j+1} \right) = \frac{-1}{n!}\left(\sum_{j = 0}^{(n-2)/2} {n \choose 2j+1} \right) = \frac{-2^{n-1}}{n!}.$$ In both of these expressions we are using the fact that ${n \choose k} = {n \choose n-k}$ to rewrite one copy of each term in the sum with coefficient $2$, and then rewrite the whole expression as a single sum. But this is just the sum of ${n \choose k}$ for $k$ even in the first case and $k$ odd in the second case, both of which equal $2^{n-1}$. So in either case, this agrees with the leading term of $\nu(t)$, as desired.

		Now, consider $\mu(0)$. The first and third sums in the expression for $\mu(0)$ cancel, as do the second and fourth sums, and so $\mu(0) = 0$. Similarly, since $E_n(0) = 0$ for $n$ even and $n\geq 2$, we see that $\nu(0) = 0$ and so $\mu$ and $\nu$ agree at $0$. The exact same argument shows that all order $k$ derivatives of both $\mu$ and $\nu$ with $k$ even and between $2$ and $n-2$ agree at $t=0$.
		
		At $t=1$, the argument is slightly more interesting. We have $\nu(1) = 1$, and for $k$ even with $2 \leq k \leq n-2$, $\nu^{(k)}(1) = 0$. So we need to see that $\mu(1) = 1$ and $\mu^{(k)}(1) = 0$ for $k$ even.
		Substituting $t=1$ into the explicit expression for $\mu(t)$, we can see that the third and fourth sums are offset exactly by a an additive factor of $2$ from the first two sums, and so the only terms that do not cancel in this whole expression are the $i=0$ and $i=1$ terms in the first sum. Thus $\mu(1)$ collapses to $$\frac{1}{n!}\left( (-1)^{n - n/2}{n \choose n/2} +\sum_{m=n/2 + 1}^n (-1)^{(n-m)} {n \choose m} \Big((m - \frac{n}{2}+1)^n + (m - \frac{n}{2} - 1)^n  \Big)\right).$$
		Now, using the fact that $n$ is even, we can rewrite this as $$\frac{(-1)^n}{n!}\sum_{j=0}^n(-1)^j{n \choose j}\Big(j - \frac{n}{2} + 1\Big)^n,$$  where the terms with $m<\frac{n}{2}$ are the second terms in the above sum. Now by Lemma~\ref{lemma25}, this is $\frac{(-1)^{2n}n!}{n!} = 1$.

		The proof for the derivatives of $\mu(t)$ at $t=1$ is, \textit{mutatis mutandis}, the same. For $k$ even with $2 \leq k \leq n-2$, $\mu^{(k)}(1)$ collapses to $$\frac{1}{(n-k)!}\left( (-1)^{n - n/2}{n \choose n/2} +\sum_{m=n/2 + 1}^n (-1)^{(n-m)}{n \choose m} \Big((m -\frac{n}{2}+1)^{n-k} + (m - \frac{n}{2} - 1)^{n-k}  \Big)\right),$$ which rearranges via an analogous argument to $$\frac{(-1)^n}{(n-k)!}\sum_{j=0}^n(-1)^j{n \choose j}\Big(j - \frac{n}{2} + 1\Big)^{n-k},$$ which is $0$ by Lemma~\ref{lemma25}.
		
		So, we have seen that all even order (between $0$ and $n-2$) derivatives of $\mu(t)$ and $\nu(t)$ agree at both $t=0$ and $t=1$ and their leading terms agree, and so these polynomials are identical. This proves point (1).

		Cases (2) and (3) are dealt in a similar way, with explicit computations; see the proof of \cite[Theorem 24]{Shideler18} for further details.
	\end{proof}

	As a corollary of the previous theorem, we recover a result of Gessel and Monsky \cite{GesselMonsky}, which they proved using different methods, although they remark that their original proof, not included in their paper, used Eulerian polynomials. We also note that, after the preprint of this article was released, two further independent proofs of this result were obtained by Meng \cite{Meng26} and Pak--Shapiro--Smirnov--Yoshida \cite{PSSY26}.
	
	\begin{corollary}[Gessel--Monsky]\label{cor:Gessel Monsky}
		Let  $f=x_1^2+\cdots+x_n^2$ and for any prime number $p$ let $e_{\HK}(A_p/f)$ be the Hilbert-Kunz multiplicity of $f$ in $A_p=\mathbb{F}_p\llbracket x_1,\dots,x_n\rrbracket$.
		Then,
		\[
		\lim_{p\to\infty}e_{\HK}(A_p/f)= 1+ c_n,
		\]
		where $c_n$ is the coefficient of $z^{n-1}$ in the power series expansion of  $\sec(z)+\tan(z)$.
	\end{corollary}
	
	\begin{proof}
		By the convergence of the derivatives of the limit functions, we have that 
		$$ \lim_{p\to\infty}e_{\HK}(A_p/f) = -\lim_{p \rightarrow \infty} \partial_{+} \psi_{f,p}(0) = -\psi_f '(0)=\phi_f'(0).$$
		We first treat the case when $n$ is even. By the previous theorem, $\phi_f ' (0) = 1 + \frac{(-1)^{\frac{n}{2}}2^{n-1}}{n!}E_n'(0)$. 
		Thus, by Lemma~\ref{lem:eulerfacts}, we have that $\phi_f ' (0)$ is
		\[ 1  + \frac{(-1)^{\frac{n}{2}}2^{n-1}}{n!}(nE_{n-1}(0)) = 1 + \frac{(-1)^{\frac{n-2}{2}}2^{n}}{n!}(2^{n} - 1)B_{n}  
		=1 + \frac{(-1)^{j}2^{2j+2}}{(2j+2)!}(2^{2j+2} - 1)B_{2j+2},
		\]
		where in the last equality we used the substitution $n-1 = 2j+1$. This last expression is precisely $1 +$ the coefficient of $z^{n-1}$ in the power series expansion of  $\sec(z)+\tan(z)$, as desired.
		For $n$ odd, we need to show that $1 + \frac{(-1)^{\frac{n-1}{2}}E_{n-1}}{(n-1)!}$ equals $\phi_f'(0) = 1 + \frac{(-1)^{\frac{n-1}{2}}2^{n-1}}{(n-1)!}(E_{n-1}(\frac{1}{2}))$, but this is immediate from  Lemma~\ref{lem:eulerfacts}  again. Thus, we have recovered the desired formulas.
	\end{proof}
	
	Using the same method, one also recovers a similar result for the limit of the $F$-signature.
	
	\begin{corollary}
		Let  $f=x_1^2+\cdots+x_n^2$ and for any prime number $p$ let $\fs(A_p/f)$ be the $F$-signature of $f$ in $A_p=\mathbb{F}_p\llbracket x_1,\dots,x_n\rrbracket$.
		Then,
		\[
		\lim_{p\to\infty}\fs(A_p/f)= 1- c_n,
		\]
		where $c_n$ is the coefficient of $z^{n-1}$ in the power series expansion of  $\sec(z)+\tan(z)$.
	\end{corollary}
	\begin{proof}
		By Theorem \ref{thm16} and Remark \ref{rk:analytic properties phi and psi}, we have $\lim_{p\to\infty}\fs(A_p/f)=\phi_f'(1)$. Then one proceeds exactly as in the proof of Corollary \ref{cor:Gessel Monsky}, applying the relations of Lemma \ref{lem:eulerfacts} to the (derivatives of the) explicit formulas of Theorem \ref{thm:squareslim}. The only difference with the computations of the proof of Corollary \ref{cor:Gessel Monsky} is a minus sign, that leads to the claimed equation.
	\end{proof}
	
	\begin{remark}\label{remarkADE} Using a similar approach, it is possible to compute the limit function $\phi_f$ for the two-dimensional singularities $A_{n-1}$ ($f=x^2+y^2+z^n$, $n\geq2$), $E_6$ ($f=x^2+y^3+z^4$), and $E_8$ ($f=x^2+y^3+z^5$); an explicit computation can be found in
		\cite{Shideler18}.
	\end{remark}

	\section{$F$-signature series and sum of hypersurfaces}\label{section:fsignatureseries}
	We fix a perfect field $\KK$ of characteristic $p>0$.
	In this chapter, we develop some techniques that will allow us to compute the $F$-signature function of a hypersurface $f$ which can be written as sums of disjoint set of variables. 
	Our main tools are the  $F$-signature series, that is the generating series
	\[ 
	\FSS_{f}(z)=\sum_{e=0}^{\infty}\FS_{f}(e)z^e,
	\]
	and the function  $\phi_{f,p}$	of \eqref{eq:defphi_p,f}. Differently from the previous sections, we are not interested in the limit function $\phi_f=\displaystyle\lim_{p\to\infty}\phi_{f,p}$, but instead we will try to obtain more information on $\phi_{f,p}$ for fixed $p$.
	Inspired by \cite{MonTexII}, we introduce the following.
	
	\begin{definition}\label{dfn:operatore bilineare r}
		Let $n\ge 2$ be a fixed integer. Let $\u=(u_0,u_1,\dots)$ and $\v=(v_0,v_1,\dots)$ be elements of $\Lambda$. We define
		\begin{equation*}
			r_n(\u,\v)=(1-p^{n-1}z)\cdot\sum_{e=0}^\infty \alpha(u_e v_e)(p^n z)^e\in \mathbb{Q}\llbracket z\rrbracket.
		\end{equation*}
	\end{definition}
	
	\noindent The following proposition summarizes some of the key properties of the function $r_n(\cdot,\cdot)$.
	
	\begin{proposition}\label{prop:properties of the function r}
		Let $n\ge 2$ be an integer and $\u,\v\in\Lambda$. Then the following facts hold.
		\begin{itemize}
			\item[(a)] The function $r_n(\cdot,\cdot)$ is $\mathbb{Z}$-bilinear.
			\item[(b)] $r_n(\u,\v)=(1-p^{n-1}z)\cdot\alpha(u_0 v_0)+p^nz\cdot r_n(S(\u),S(\v))$.
			\item[(c)] Assume that, for all $e\in\mathbb{N}$, $u_e$ and $v_e$ are a linear combination of $\lambda_i$ with $i<p^e$. Then, for $0\le i,j <p$, we have
			\begin{equation*}
				r_n(\lambda_i\u,\lambda_j\v)=\begin{cases}
					r_n(\u,\v) &\text{if $i=j$}\\
					0 &\text{if $i\ne j$.}
				\end{cases}
			\end{equation*}
			\item[(d)] If $\u\in\Lambda_0$, then $r_n(\u,\Delta)=\alpha(u_0)$.
			\item[(e)] $r_n(R(\u),R(\v))=r_n(\u,\v)$.
		\end{itemize}
	\end{proposition}
	\begin{proof}
		Point (a) is a consequence of the $\mathbb{Z}$-linearity of $\alpha$ and point (b) descends directly from the definition of $r_n(\cdot,\cdot)$.
		
		In order to prove (c), notice that
		\begin{equation*}
			\begin{split}
				\alpha((\lambda_i \u)_e(\lambda_j \v)_e)=\alpha(\theta^e(\lambda_i)u_e\theta^e(\lambda_j)v_e)=\alpha(\theta^e(\lambda_i\lambda_j)u_ev_e)
			\end{split}
		\end{equation*}
		since $\theta$ is a ring homomorphism. The fact that $\Gamma_e$ is a subring of $\Gamma$  implies that $u_e v_e$ is a linear combinations of elements $\lambda_r$ with $r<p^e$.
		On the other hand, if $i=j$ then there exists an $M>0$ such that
		\begin{equation*}
			\theta^e(\lambda_i\lambda_j)=\theta^e(\sum_{i=0}^M\lambda_i)=\lambda_0+\lambda_{2p^e-1}+\dots,
		\end{equation*}
		where all indices other that $0$ in the sum are greater or equal than $2p^e-1\ge p^e$. The fact that $\alpha(\lambda_i\lambda_j)=\delta_{i,j}$ implies that $\alpha(\theta^e(\lambda_i\lambda_j)u_ev_e)=\alpha(u_e v_e)$.
		If $i\ne j$ then there exist $N,M>0$ such that
		\begin{equation*}	     \theta^e(\lambda_i\lambda_j)=\theta^e(\sum_{i=N}^M\lambda_i)=\lambda_{i_0}+\lambda_{i_1}+\dots,
		\end{equation*}
		where all indices in the sum are greater or equal that $2p^e-1\ge p^e$. Again, the fact that $\alpha(\lambda_i\lambda_j)=\delta_{i,j}$ implies that $\alpha(\theta^e(\lambda_i\lambda_j)u_ev_e)=0$.
		
		Point (d) can be easily proved by observing that $\u\cdot\Delta=\alpha(u_0)\Delta$ (see Example \ref{ex:properties of Delta}), hence
		\begin{equation*}
			\alpha(u_e\Delta_e)=\alpha(u_0)\alpha(\Delta_e)=p^{-e}\alpha(u_0).
		\end{equation*}
		
		Point (e) can be checked directly using the definition of $r_n$ (see Definition \ref{dfn:operatore bilineare r}), the definition of the operator $R$ (see equation \eqref{eq:definition reflection and shifting}) and the multiplication rules in \eqref{eq:formule moltiplicazione lambda}.
	\end{proof}
	
	The main connection between the operator $r_n(\cdot,\cdot)$ and the $F$-signature is given by the following theorem.
	
	\begin{theorem}\label{lemma:chesaraimportante}\label{cor:from function r to FSS}
		Let $1\le r\le n$ be integers, $f\in \KK\llbracket x_1,\dots,x_{n-r}\rrbracket$ and $g\in \KK\llbracket x_{n-r+1},\dots,x_n\rrbracket$ be power series such that $\phi_{f,p}$ and $\phi_{g,p}$ are $p$-fractals.
		
		Set $\u=\mathscr{L}(\phi_{f,p})$ and $\v=\mathscr{L}(\bar{\phi}_{g,p})$. Then 
		\begin{itemize}
			\item[(a)] $\mathscr{L}(\phi_{f+g,p})=\mathscr{L}(\phi_{f,p})\mathscr{L}(\phi_{g,p})$.
			
			\item[(b)] The $F$-signature series of the hypersurface $f+g$ in $\KK\llbracket x_1,\dots,x_{n}\rrbracket$ is given by
			\begin{equation*}
				\FSS_{f+g}(z)=\frac{r_n(\u,\v)}{p^{n-1}z-1}.
			\end{equation*}
		\end{itemize}

	\end{theorem}
	
	\begin{proof}
		Point (a) is proven in the paragraph after \cite[Definition~3.2]{MonTexII}, so we skip directly to the proof of (b). By  Lemma~\ref{lemma_$F$-signaturefractal}, we have 
		\begin{equation}\label{eq:3515}
			\frac{1}{p^{ne}}\FS_{f+g}(e)=1-\phi_{f+g,p}\left(1-\frac{1}{p^e}\right)= 1-\bar{\phi}_{f+g,p}\Big(\frac{1}{p^e}\Big),
		\end{equation}
		where $\bar{\phi}_{f+g,p}$ is the reflection of $\phi_{f+g,p}$ as in \eqref{eq:reflection}.
		Using the operator $\mathscr{L}$ of \eqref{eq:definizione serie L}, the linear function  $\alpha\left(\sum_{i=0}^nc_i\lambda_i\right)= c_0$ and the reflection operator $R$ of \eqref{eq:definition reflection and shifting}, we can write \eqref{eq:3515} as
		\begin{equation}\label{eq:3524}
			1-\bar{\phi}_{f+g,p}\Big(\frac{1}{p^e}\Big)=-\alpha\Big(\mathscr{L}(\bar{\phi}_{f+g,p})_e\Big)=-\alpha\Big(R\big(\mathscr{L}(\phi_{f+g})\big)_e\Big).
		\end{equation}
		Applying \eqref{eq:multipliction rule reflection} together with point (a), the quantity in \eqref{eq:3524} becomes
		\begin{equation*}
			\begin{split}
				-\alpha\Big(R(\mathscr{L}(\phi_{f+g}))_e\Big)&=-\alpha\Big(R\big(\mathscr{L}(\phi_{f,p})\mathscr{L}(\phi_{g,p})\big)_e\Big)=-\alpha\Big(R\big(\mathscr{L}(\phi_{f,p})\big)_e\mathscr{L}(\phi_{g,p})_e\Big)\\
				&=-\alpha\Big(\big(\mathscr{L}(\bar{\phi}_{f,p})\mathscr{L}(\phi_{g,p})\big)_e\Big).
			\end{split}
		\end{equation*}
		Putting everything together, we obtain that the $F$-signature function of $f+g$ is given by
		\[
		\FS_{f+g}(e)=-p^{ne}\alpha\Big(\big(\mathscr{L}(\bar{\phi}_{f,p})\mathscr{L}(\phi_{g,p})\big)_e\Big).
		\]
		Thus, the properties of Proposition~\ref{prop:properties of the function r} yield the desired formula.
	\end{proof}

	\begin{example}
		\label{ex:first shifting rules}
		Let $d\ge 2$ be an integer smaller than $p$ and consider $x^d\in\KK\llbracket x\rrbracket$ with corresponding function $\phi_{x^d,p}$. The aim of this example is to compute, when $p\equiv \pm 1 \ \mathrm{mod} \ d$, an explicit relation (called \emph{shifting rule}) for $S(\mathscr{L}(\phi_{x^d,p}))$ that will be the starting point of all our work of the next section.
		
		For $e\ge 1$ and $0\le a\le p^e$, an explicit computation shows that
		\begin{equation*}
			\phi_{x^d,p}\Big(\frac{a}{p^e}\Big)=p^{-e}\dim_\KK\left(\KK\llbracket x\rrbracket/(x^{p^e},x^{ad}\right)=p^{-e}\min\{p^e,ad\}=\min\Big\{1,\frac{ad}{p^e}\Big\},
		\end{equation*}
		that is
		\begin{equation*}
			\phi_{x^d,p}(t)=\min\{1,dt\}
		\end{equation*}
		for every $t\in\mathscr{I}=[0,1]\cap\left\{\frac{a}{p^e}: \ a,e\in\N\right\}$. Define also
		\begin{equation*}
			M=\Big\lfloor\frac{p}{d}\Big\rfloor=\begin{cases}
				\frac{p-1}{d} &\text{if $p\equiv 1 \ \mathrm{mod} \ d$}\\
				\frac{p-(d-1)}{d} &\text{if $p\equiv -1 \ \mathrm{mod} \ d$}.
			\end{cases}
		\end{equation*}
		Let now $p\equiv 1 \ \mathrm{mod} \ d$. An explicit computation gives 
		\begin{equation}\label{eq:calcolo immagine operatori T}
			\begin{split}
				p(T_{p|i}\phi_{x^d,p})(t)=\min\{p,dt+di\}=\begin{cases}
					dt+di &\text{if $i<M$}\\
					\phi_{x^d,p}(t)+p-1 &\text{if $i=M$}\\
					p &\text{if $i>M$}
				\end{cases}
			\end{split}	
		\end{equation}
		for every $t\in\mathscr{I}$ and $0\le i < p$. Then, using \eqref{eq:definizione serie L}, we can compute the first two terms of $\mathscr{L}(\phi_{x^d,p})$, that are $\mathscr{L}(\phi_{x^d,p})_0=\lambda_0$ and
		\begin{equation*}
			\begin{split}
				\mathscr{L}(\phi_{x^d,p})_1=\sum_{i=0}^{p-1}\Big(\min\Big\{1,\frac{di+d}{p}\Big\}-\min\Big\{1,\frac{di}{p}\Big\}\Big)(-1)^i\lambda_i=(-1)^M\frac{1}{p}\lambda_M+\frac{d}{p}\Big(\sum_{i=0}^{M-1}(-1)^i \lambda_i\Big).
			\end{split}
		\end{equation*}
		Notice also that the first entry of $R(\mathscr{L}(\phi_{x^d,p}))=\mathscr{L}(\bar{\phi}_{x^d,p})$ is $-\lambda_0$, where $R$ is the reflection operator defined in \eqref{eq:definition reflection and shifting}. Calling $\u=\mathscr{L}(\phi_{x^d,p})$, we obtain that
		\begin{equation*}
			pS(\u)-\lambda_MR^M(\u)=\Big(d\sum_{i=0}^{M-1}(-1)^i\lambda_i,\dots\Big).
		\end{equation*}
		On the other hand, using the description of the shifting operator $S$ given in \eqref{eq:scrittura esplicita shifting operator} together with the explicit formula of \eqref{eq:calcolo immagine operatori T} and noting that $\mathscr{L}(\text{constant})=0$ and $\mathscr{L}(t)=\Delta$ (see also Example \ref{ex:properties of Delta}), we deduce that
		\begin{equation*}
			pS(\u)=\lambda_M R^M(\u)+(\text{linear comb. of $\lambda_i\Delta$ with $i<M$}).
		\end{equation*}
		Putting together the last two equations, we obtain the \emph{shifting rule}
		\begin{equation}\label{eq:shifting rule p congruo a 1}
			pS(\u)=\lambda_M R^M(\u)+d\sum_{i=0}^{M-1}(-1)^i\lambda_i\Delta.
		\end{equation}
		Let now $p\equiv -1 \ \mathrm{mod} \ d$. Following step by step the argument for the case $p\equiv 1 \ \mathrm{mod} \ d$, we first find that
		\begin{equation*}
			\begin{split}
				p(T_{p|i}\phi_{x^d,p})(t)=\min\{p,dt+di\}=\begin{cases}
					dt+di &\text{if $i<M$}\\
					\bar{\phi}_{x^d,p}(t)+p+d(t-1) &\text{if $i=M$}\\
					p &\text{if $i>M$}
				\end{cases}
			\end{split}	
		\end{equation*}
		for every $t\in\mathscr{I}$ and $0\le i < p$. We find also that $\mathscr{L}(\phi_{x^d,p})_0=\lambda_0$ and
		\begin{equation*}
			\begin{split}
				\mathscr{L}(\phi_{x^d,p})_1=(-1)^M\frac{d-1}{p}\lambda_M+\frac{d}{p}\Big(\sum_{i=0}^{M-1}(-1)^i \lambda_i\Big).
			\end{split}
		\end{equation*}
		Calling $\u=\mathscr{L}(\phi_{x^d,p})$, we obtain that
		\begin{equation*}
			pS(\u)-\lambda_MR^{M+1}(\u)=\Big(d\sum_{i=0}^{M}(-1)^i\lambda_i,\dots\Big).
		\end{equation*}
		Again, using the formula in \eqref{eq:scrittura esplicita shifting operator} we obtain the \emph{shifting rule}
		\begin{equation}\label{eq:shifting rule p congruo a -1}
			pS(\u)=\lambda_MR^{M+1}(\u)+d\sum_{i=0}^{M}(-1)^i\lambda_i\Delta.
		\end{equation}
	\end{example}
	
	We conclude the section by recalling the following classical result about generating series which will be useful to compute the $F$-signature function from the $F$-signature series. For a proof see also \cite[Proposition~2.3]{CZ23}.
	
	\begin{proposition}\label{prop:successione-serie}
		Let $\{t_m\}_{m\in\N}$ be a sequence of rational numbers and let $\delta_0,\dots,\delta_d\in L\setminus\{0\}$ be distinct elements, with $L$ an extension field of $\Q$. The following statements are equivalent:
		\begin{itemize}
			\item[(a)] There exist $a_0,\dots,a_d\in L$ such that
			
			\begin{equation*}
				t_m=a_0 \delta_0^m+a_{1}\delta_1^{m}+\dots +a_d\delta_d^m
			\end{equation*}
			for every $m\in\N$.
			\item[(b)] There exist $P(z),Q(z)\in L[z]$ such that
			\begin{equation*}
				\sum_{m=0}^{\infty}t_mz^m= \frac{P(z)}{Q(z)}
			\end{equation*}
			with $Q(z)=\displaystyle\prod_{i=0}^d(1-\delta_i z)$ and $\deg P(z)<\deg Q(z)=d+1$.
		\end{itemize}
	\end{proposition}
	
	In the previous proposition, the explicit way to pass from the generating series to the sequence is via the partial fraction decomposition.

	\section{$F$-signature function of Fermat hypersurfaces}\label{section:Fermat}
	
	In this section we investigate the shape of the $F$-signature function of Fermat hypersurfaces.
	We fix the following notation.
	
	\begin{notation}\label{notation:Fermat}
		Let $n,d\geq2$ be integers and let $\KK$ be a perfect field of characteristic $p>0$. We consider the Fermat hypersurface
		\[
		f=x_1^d+\cdots+x_n^d
		\]
		of degree $d$ in $n$ variables in the power series ring $A_p=\KK\llbracket x_1,\dots,x_n\rrbracket$.
	\end{notation}
	
	When the quotient ring $A_p/f$ is not $F$-pure, then the $F$-signature function $\FS_{A_p/f}$ is identically zero. The F-purity of $A_p/f$ (depending on $p$, $d$, and $n$) was established by Hochster and Roberts \cite[Proposition~5.21]{HR76}.
	The situation is as follows:
	\begin{itemize}
		\item If $p\leq d$ or $d>n$, then $A_p/f$  is not $F$-pure;
		\item If $d=n$, then $A_p/f$  is $F$-pure if and only if $p\equiv1 \ \mathrm{mod} \ d$;
		\item If $d<n$, then there exists an integer $v$ depending only on $d,n$ such that if $p\geq v$, then $A_p/f$ is $F$-pure.
	\end{itemize}
	
	Therefore, from now on we will assume $d\leq n$ and $d<p$. We observe also that by \cite[Corollary~3.6.14]{BH98}, the a-invariant of $A_p/f$  is $\mathrm{a}(A_p/f)=d-n$. So, in the case $d=n$ the a-invariant is zero, therefore the ring $A_p/f$ is not strongly $F$-regular by \cite[Exercise~10.3.28]{BH98}, thus $\fs(A_p/f)=0$. In this case, we can further describe the $F$-signature function of $f$ as follows.
	
	\begin{proposition}\label{prop:$F$-signature function s=d}
		Let $d= n\ge 2$ be integers such that $d<p$ and $p\equiv1 \ \mathrm{mod} \ d$. Call $f=x_1^d+\cdots+x_d^d$. Then, we have $\FS_{A_p/f}(e)=1$ for every $e\geq0$.
	\end{proposition}
	
	\begin{proof}
		By Fedder/Glassbrenner's Criterion \cite{Fedder83, Glassbrenner96}, the hypersurface $A_p/f$ is readily seen to be $F$-pure but not\footnote{One can also appeal to the fact that $A_p/f$ is not log terminal to see that it is not strongly $F$-regular.} strongly $F$-regular. As all uniformly $F$-compatible ideals in an $F$-pure ring are radical \cite{SchwedeCentersofFpurity} and $f$ has an isolated singularity, the test ideal of the associated hypersurface (which is the smallest uniformly $F$-compatible ideal containing a nonzerodivisor and determines the non-strongly $F$-regular locus) must be the maximal ideal itself. However, this further implies that the test ideal agrees with the $F$-splitting prime of the hypersurface (the largest proper uniformly $F$-compatible ideal). Thus, we see that the splitting dimension is $\mathrm{sdim}(A_p/f)=0$, and in this case $\FS_{A_p/f}(e)=1$ for every $e\geq0$ by \cite[Corollary~3.4]{AE05}. Indeed, the splitting ideals are proper and contain the splitting prime, hence must also all be maximal. Since $\KK$ is perfect, each splitting number is therefore $1$.
	\end{proof}
	
	\begin{remark}
		One can also obtain Proposition~\ref{prop:$F$-signature function s=d} using the machinery developed in Section~\ref{section:fsignatureseries} and computations similar to those in the proof of the forthcoming Theorem~\ref{thm:shapeFermat}. 
	\end{remark}

	Now, we focus on the case $d<n$. We recall that for $p\gg0$, we have $\mathrm{a}(A_p/f)<0$, so the ring $A_p/f$ is strongly $F$-regular by \cite[Theorem~7.12]{HH94} (see also \cite{Hoc22}), therefore, $\fs(A_p/f)\neq0$.
	More generally, we can describe the $F$-signature function of $f$ as follows.

	\begin{theorem}\label{thm:shapeFermat}
		Let $n>d\ge2$ be integers with $p>d$ and $p\equiv\pm1\ \mathrm{mod}\ d$ and assume that 
		\begin{equation*}
			\begin{cases}
				n\cdot\left(\frac{p-1}{d}\right)  &\text{is even if $p\equiv 1\ \mathrm{mod}\ d$}\\
				n\cdot\left(\frac{p+1}{d}\right)  &\text{is even if $p\equiv -1\ \mathrm{mod}\ d$}.
			\end{cases}
		\end{equation*}
		Then, there exist constants $B,C\in\mathbb{Z}$ (depending on $p,d,n$) with $0\le B\le p^{n-3}$ such that the $F$-signature function of $f=x_1^d+\cdots+x_n^d$ is
		\[
		\FS_{f}(e)= \fs(f)\cdot p^{(n-1)e}+(1-\fs(f))\cdot B^e 
		\]
		with $\fs(f)=\displaystyle\frac{-C}{p^{n-1}-B}\in\mathbb{Q}$.
	\end{theorem}
	\begin{proof}
		Let $M=\big\lfloor \frac{p}{d}\big\rfloor$, $\u=\mathscr{L}(\phi_{x_1^d+\dots+x_{n-1}^d,p})$ and $\v=\mathscr{L}(\bar{\phi}_{x^d,p})$. Let also
		\begin{equation*}
			\alpha=\begin{cases}0 \ &\text{if } p\equiv1\ \mathrm{mod} \ d \\ 1&\text{if } p\equiv-1\ \mathrm{mod} \ d.
			\end{cases}
		\end{equation*}
		
		Starting from the shifting rules \eqref{eq:shifting rule p congruo a 1} and \eqref{eq:shifting rule p congruo a -1}, taking the $(n-1)$-fold product and using the multiplication rules in \eqref{eq:formule moltiplicazione lambda}, point (a) of Theorem \ref{lemma:chesaraimportante} and the properties of $\Delta$ gives
		\begin{equation*}
			p^{n-1}S(\u)=\lambda_M^{n-1}R^{(n-1)(M+\alpha)}(\u) +\text{\Big(lin. comb. of $\lambda_i\Delta$\Big)}_{i\le(n-1)M}
		\end{equation*}
		and by applying the reflection operator $R$ to the one-variable shifting rule together with the commutation relation \eqref{eq:rule between R and S} we have
		\begin{equation}
			p S(\v)=\lambda_{p-1-M}R^{M+\alpha}(\v)+\text{\Big(lin. comb. of $\lambda_i\Delta$ \Big)}_{p-1-M\le i\le p-1}
		\end{equation}
		for any $p\equiv \pm 1\ \mathrm{mod}\ d$. Using point (b) of Proposition \ref{prop:properties of the function r} we find that
		\begin{equation}\label{eq:r computed when s>d}
			r_n(\u,\v)=(1-p^{n-1}z)\cdot (-1)+z\cdot r_n\big(p^{n-1}S(\u),pS(\v)\big).
		\end{equation}
		
		Substituting the appropriate shifting rules, the properties of the function $r_n$ (see Proposition \ref{prop:properties of the function r}) together with the properties of the sequence $\Delta$ (see Example \ref{ex:properties of Delta}) yield that
		\begin{equation*}
			r_n(\u,\v)=p^{n-1}z-1+z\cdot \Big(B\cdot r_n\big(R^{(n-1)(M+\alpha)}(\u),R^{M+\alpha}(\v)\big)+C\Big)
		\end{equation*}
		for some constants $B,C\in\mathbb{Z}$. Notice that, by definition, the constant $B$ coincides with the coefficient of $\lambda_{p-1-M}$ in the power $\lambda_M^{n-1}$. The estimate \eqref{eq:stima coefficienti potenza lambda} on the coefficients of $\lambda_M^{n-1}$ yields that $0\le B\le p^{n-3}$. 
		
		The assumption on the parity of $nM$ and $n(M+1)$ imply that $M+\alpha$ and $(n-1)(M+\alpha)$ have the same parity, therefore point (e) of Proposition \ref{prop:properties of the function r} yields
		\begin{equation*}
			r_n(\u,\v)=p^{n-1}z-1+z\cdot (B\cdot r_n(\u,\v)+C).
		\end{equation*}
		Applying point (b) of Theorem~\ref{cor:from function r to FSS} we obtain that
		\begin{equation*}
			\FSS_{f_{d,n}}(z)=\frac{-1+(p^{n-1}+C)z}{(p^{n-1}z-1)(1-Bz)}.
		\end{equation*}
		Then, Proposition~\ref{prop:successione-serie} and partial fraction decomposition lead the claimed $F$-signature function.
	\end{proof}
	
	We focus on the study of the constants $B,C$ that appear in the statement of Theorem \ref{thm:shapeFermat}, in order to obtain more precise information on the $F$-signature function of Fermat hypersurfaces.
	
	\begin{proposition}\label{prop:comportamento di B e C}
		Keep notations and assumptions as in Theorem \ref{thm:shapeFermat} and suppose $n\ge 4$.
		\begin{itemize}
			\item[(i)] If $p\equiv 1 \ \mathrm{mod}\ d$ then $B>1$.
			\item[(ii)] If $p\equiv -1 \ \mathrm{mod}\ d$ then $B\ge 1$ if and only if
			\begin{equation}\label{eq:disuguaglianza B=0}
				p(n-d)+n+d-dn\ge 0,
			\end{equation}
			and $B=1$ if and only if equality holds in \eqref{eq:disuguaglianza B=0}. Otherwise, we have that $B=0$ and $C=0$.
			\item[(iii)] If $B=1$ and $d$ is odd, then 
			\begin{equation*}
				\fs(f)=\frac{(d-1)^n-1}{p^{n-1}-1}.
			\end{equation*}
		\end{itemize}
	\end{proposition}
	\begin{proof}
		Let $M=\lfloor\frac{p}{d}\rfloor$. During the proof of Theorem \ref{thm:shapeFermat}, we noticed that $B$ coincides with the coefficient of $\lambda_{p-1-M}$ in $\lambda_M^{n-1}$.
		
		(i) When $p\equiv 1\ \mathrm{mod}\ d$, using the last  multiplication formula in \eqref{eq:formule moltiplicazione lambda} and the fact that $n\ge\max\{4,d+1\}$, one can easily see that the coefficient of $\lambda_{p-1-M}$ in $\lambda_M^{n-1}$ is always greater then $1$, hence in this case $B>1$.
		
		(ii) When $p\equiv -1\ \mathrm{mod}\ d$, the same argument implies that $B\ne 0$ if and only if $(n-1)M\ge p-1-M$, i.e. when
		\begin{equation*}
			p(n-d)+n+d-dn\ge 0,
		\end{equation*}
		and $B=1$ if and only if equality holds. Notice also that when $B=0$ we have that $(n-1)M< p-1-M$, therefore we obtain that also $C=0$.
		
		(iii) Suppose now that $B=1$ and $d$ odd, so that necessarily $p\equiv -1 \ \mathrm{mod}\ d$ and $M$ is odd. By the proof of point (ii), we have $(n-1)M=p-1-M$, and this implies that $n$ must be even. Follow the proof of Theorem \ref{thm:shapeFermat} and keep its notation. We can write the relevant shifting rules as
		\begin{equation*}
			\begin{split}
				p^{n-1}S(\u)&=\lambda_M^{n-1} \u + a_{n-1}\lambda_{p-1-M}\Delta+\text{\Big(lin. comb. of $\lambda_i\Delta$\Big)}_{i<p-1-M} \quad \text{and}\\
				p S(\v)&=\lambda_{p-1-M}\v-d\sum_{i=0}^{M}(-1)^i\lambda_{p-1-i}\Delta
			\end{split}
		\end{equation*}
		for some $a_{n-1}\in\mathbb{Z}$. The properties of the function $r_n$ (see Proposition \ref{prop:properties of the function r} point (c)) imply that only the coefficients of $\lambda_{p-1-M}$ give a contribution for the constant $C$, and in particular
		\begin{equation}\label{eq:formula per C quando B=1}
			C=da_{n-1}-a_{n-1}+d=(d-1)a_{n-1}+d.
		\end{equation}
		Applying the last formula in \eqref{eq:formule moltiplicazione lambda} we easily find that the number $a_{n-1}$, that comes from the product of \eqref{eq:shifting rule p congruo a -1} with itself $n-1$-times, is the $n-1$-th term of the sequence $\{a_e\}_{e\ge1}$ defined as
		\begin{equation*}
			\begin{cases}
				a_1=-d\\
				a_{e+1}=(1-d)a_e-d &\text{for $e\ge 1$}.
			\end{cases}
		\end{equation*}
		Elementary considerations allow us to conclude that indeed $a_e=(1-d)^e-1$. Coming back to equation \eqref{eq:formula per C quando B=1} we find that
		\begin{equation*}
			C=-(d-1)^n+1
		\end{equation*}
		and Theorem \ref{thm:shapeFermat} leads the claimed $F$-signature.
	\end{proof}
	
	\begin{remark}\label{rem:considerationsonBandC}
		Some considerations on the previous proposition:
		\begin{itemize}
			\item[(a)] If $d=2$ and $n=3$ then $B=1$ for every prime $p$ that satisfies the conditions of the theorem. Indeed, it is known that in this case the $F$-signature function is $\FS_f(e)=\frac{1}{2}p^{2e}+\frac{1}{2}$. See \cite[Example~18]{HunLeu}.
			\item[(b)] When $B,C=0$ the $F$-signature function of $f$ is identically zero  and the ring $A_p/f$ is not $F$-pure. 
			\item[(c)] Point (ii) of Proposition \ref{prop:comportamento di B e C} implies that, for fixed $d$ and $p\equiv -1\ \mathrm{mod}\ d$, there is an $n_0>d$ such that $B,C=0$ for every $d<n<n_0$ and $B\ne 0$ for every $n\ge n_0$. Similarly, for fixed $d$ and $n$, there is a prime $p_0$ such that $B\ne 0$ if $p\ge p_0$ and $B,C=0$ if $p<p_0$ and $p\equiv -1\ \mathrm{mod}\ d$. In the following table we present some examples of this behaviour when $n=d+1$:
			
			\begin{center}
				\def\arraystretch{1.5}
				\begin{tabular}{c|c|c|c|c}
					& $d=\mathbf{3}$ & $d=\mathbf{4}$ & $d=\mathbf{7}$ &$d=\mathbf{21}$\\
					\hline
					$B,C=0$ &Never  &$p=7$  &$p=13$  &$p=41,83,167,251,293$\\
					\hline
					$B=1$ &$p=5$  & Never  &$p=41$ &$p=419$\\
					\hline
					$B>1$ &$p>5$  &$p>11$  &$p>41$ &$p>419$
				\end{tabular}
			\end{center}
			
			\item[(d)]  Point (ii) of Proposition \ref{prop:comportamento di B e C} yields that $B=1$ if and only if
			\begin{equation}\label{eq:p quando B=1}
				p=\frac{dn-n-d}{n-d}.
			\end{equation}
			The number of primes $p\equiv -1\ \mathrm{mod}\ d$ representable in this way is unknown, but some explicit computations and some conjectures suggest that it might be infinite (see also Remark \ref{rk:primes representable as quadratic polynomials}).
			In any case, let $p_0 \equiv -1\ \mathrm{mod}\ d$ be a prime representable as in \eqref{eq:p quando B=1} for some $d\ge 3$ odd and some $n>d$ (necessarily even). Then, Proposition \ref{prop:comportamento di B e C} implies that the Fermat hypersurface $f$ is strongly $F$-regular for every prime $p \equiv \pm 1\ \mathrm{mod}\ d$ with $p\ge p_0$ and it is not $F$-pure for every prime $p \equiv  -1\ \mathrm{mod}\ d$ with $p<p_0$.
		\end{itemize}
	\end{remark}

	It seems not an easy task to determine explicitly the constants $B,C$ of Theorem~\ref{thm:shapeFermat} when $B>1$. The first non-trivial case is when $n=d+1$. 
	
	\subsection{A question by Watanabe and Yoshida}
	In \cite[Proposition~2.4]{WatYosh}, Watanabe and Yoshida prove the following inequality for the $F$-signature of the Fermat hypersurface $f=x_1^d+\cdots+x_{d+1}^d$ of degree $d$ in $n=d+1$ variables in characteristic $p>d$:
	\begin{equation}\label{eq:watyos}
		\fs(f)\leq \frac{1}{2^{d-1}(d-1)!}.
	\end{equation}
	Consequently, they ask the following natural question \cite[Question~2.6]{WatYosh}.
	\begin{question}[Watanabe-Yoshida]\label{question:watyos}
		Let $d\geq2$ be an integer and let $p>d$. Does equality hold in \eqref{eq:watyos}?
	\end{question}
	
	For $d=2$, it is well known that the previous question has an affirmative answer, since $\fs(f)=\frac{1}{2}$ (see Remark~\ref{rem:considerationsonBandC}).
	
	\begin{remark}
		For $d=3$, the upper bound in \eqref{eq:watyos} gives $\fs(f)\leq \frac{1}{8}$. Using the techniques that we have developed, we will prove in Theorem~\ref{thm:FsignatureFermatcubic} that 
		\[
		\fs(f)=\frac{3p(p-1)(p+1)}{8(3p^3-p\pm2) }=\frac{1}{8}-\frac{2(p\pm 1)}{8(3p^3-p\pm 2)}<\frac{1}{8}
		\]
		for every prime $p>3$, thus providing a negative answer to Question~\ref{question:watyos}.
	\end{remark}	
	We are also able to provide a negative answer to Question~\ref{question:watyos} for a conjecturally infinite number of $d>3$.
	
	\begin{proposition}\label{prop:WYnegative}
		Let $d>3$ be an odd integer and assume that $p=d^2-d-1$ is the characteristic of the ring $A_p=\KK\llbracket x_1,\dots,x_{d+1}\rrbracket$. Then, the $F$-signature of $f=x_1^d+\cdots+x_{d+1}^d$ is
		\begin{equation*}
			\fs(f)<\frac{1}{2^{d-1}(d-1)!}.
		\end{equation*}
	\end{proposition}
	\begin{proof}
		First notice that our assumptions are compatible with the conditions of Theorem \ref{thm:shapeFermat}. The assumption that $p=d^2-d-1$ implies that $p\equiv -1\ \mathrm{mod}\ d$ and that equation \eqref{eq:p quando B=1} is satisfied, therefore the constant $B$ appearing in Theorem \ref{thm:shapeFermat} is equal to $1$. By applying point (iii) of Proposition \ref{prop:comportamento di B e C} we obtain that
		\begin{equation*}
			\fs(f)=\frac{(d-1)^{d+1}-1}{(d^2-d-1)^d-1}.
		\end{equation*}
		Using this formula, one can check by hand that our claimed inequality is true for $d=5$. Suppose that $d\ge 7$. Calling $t=d-1$ one finds that
		\begin{equation*}
			\fs(f)=\frac{t^{t+2}-1}{(t^2+t-1)^{t+1}-1}<\frac{t^{t+2}-1}{t^{2t+2}-1} < \frac{1}{t^t}.
		\end{equation*}
		The inequality
		\begin{equation*}
			\frac{1}{t^t}\le \frac{1}{2^{t}t!}
		\end{equation*}
		is satisfied for every $t\ge 6$, yielding the claim.
	\end{proof}
	
	\begin{remark}\label{rk:primes representable as quadratic polynomials}
		The previous lemma gives a negative answer to Question~\ref{question:watyos} for every odd integer $d>3$ with the property that $d^2-d-1$ is prime. The number of primes representable as quadratic polynomials in one variable is one of the most important open problems in number theory. However, Bunyakovsky's conjecture (presented in \cite{Boun57})  suggests that the set of primes representable as $d^2-d-1$ should be infinite. We pursued some numerical computations and found out that there are 69625 primes representable as $d^2-d-1$ with $d$ odd and $3<d < 10^6$, starting from $d=5,7,9,11,17,21,25,27,29,31,\dots$.
	\end{remark}
	
	On the other hand, we can prove that Watanabe-Yoshida bound of \eqref{eq:watyos} is asymptotically sharp. This is a consequence of Theorem~\ref{thm16} and  Corollary~\ref{cor:Fsignaturediagonal}.
	
	\begin{proposition}\label{cor:WYupperboundlimit}
		Let  $f=x_1^d+\cdots+x_{d+1}^d$ and for any prime number $p$ let $\fs(A_p/f)$ be the $F$-signature of $f$ in $A_p=\mathbb{F}_p\llbracket x_1,\dots,x_{d+1}\rrbracket$.
		Then,
		\[
		\lim_{p\to\infty}\fs(A_p/f)= \frac{1}{2^{d-1}(d-1)!}.
		\]
	\end{proposition}
	
	\begin{proof}
		From the formula of Corollary~\ref{cor:Fsignaturediagonal}, we obtain
		\[
		\lim_{p\to\infty}\fs(A_p/f)= \frac{d^{d}}{2^{d+1}(d-1)!}\left(B_0+2\sum_{\lambda\in\mathbb{Z}_{\geq1}}B_\lambda\right).
		\]
		We compute the coefficients $B_\lambda$.
		First, observe that the condition $\epsilon_0+\frac{\epsilon_1}{d}+\cdots+\frac{\epsilon_{d+1}}{d}-2\lambda\geq0$ implies that $B_\lambda=0$ for all $\lambda\geq2$.
		This is because the value $\frac{\epsilon_1}{d}+\cdots+\frac{\epsilon_{d+1}}{d}=\frac{1}{d}(\epsilon_1+\cdots+\epsilon_{d+1})\leq\frac{d+1}{d}$ is always strictly smaller than $2$. So, for $\lambda\geq2$ the sum $\epsilon_0+\epsilon_1\frac{1}{d}+\cdots+\epsilon_{d+1}\frac{1}{d}-2\lambda$ is always negative.
		
		For the same reason, when $\lambda=1$ the only possible choice for the $\epsilon_i$'s such that the condition $\epsilon_0+\frac{\epsilon_1}{d}+\cdots+\frac{\epsilon_{d+1}}{d}-2\lambda\geq0$ is satisfied is $\epsilon_0=\epsilon_1=\cdots=\epsilon_{d+1}=1$. Therefore, we have
		\[
		B_1=\sum(\epsilon_1\cdots\epsilon_{d+1})\left(\epsilon_0+\frac{\epsilon_1}{d}+\cdots+\frac{\epsilon_{d+1}}{d}-2\right)^{d}=\left(1+\frac{d+1}{d}-2\right)^d=\frac{1}{d^d}.
		\]
		Finally, we show that $B_0=2/d^d$. To see this, notice that the relation
		\begin{equation*}
			\epsilon_0+\frac{\epsilon_1}{d}+\dots+\frac{\epsilon_{d+1}}{d}=\epsilon_0+\frac{\epsilon_1+\dots+\epsilon_{d+1}}{d}\ge 0
		\end{equation*}
		holds only in the following cases:
		\begin{itemize}
			\item When $\epsilon_0=1$ and at least one of the $\epsilon_i$ with $1\le i\le d+1$ is equal to $1$;
			\item When $\epsilon_0=-1$ and $\epsilon_i=1$ for every $1\le i\le d+1$.
		\end{itemize}
		Collecting the summands that appear in the definition of $B_0$ with respect to the number of $\epsilon_i$ being $-1$ and doing some combinatorics, we find that
		\begin{equation*}
			B_0=\sum_{i=0}^d (-1)^i \binom{d+1}{i}\Big(1+\frac{d+1-2i}{d}\Big)^d+\frac{1}{d^d}.
		\end{equation*}
		Then, proving that $B_0=2/d^d$ is equivalent to showing that
		\begin{equation*}
			\sum_{i=0}^{d+1} (-1)^i \binom{d+1}{i}(2d+1-2i)^d=0,
		\end{equation*}
		that is our new claim to proof. Notice that the binomial formula leads the identity
		\begin{equation*}
			\frac{(1-x^2)^{d+1}}{x}=\sum_{i=0}^{d+1}(-1)^{d+1-i}\binom{d+1}{i}x^{2d+1-2i}
		\end{equation*}
		as rational functions in the variable $x$. Let now $E=x\frac{d}{dx}$ be the operator that derives and then multiplies by $x$. Applying it $d$ times to the previous relation, we find that
		\begin{equation*}
			E^d\bigg(\frac{(1-x^2)^{d+1}}{x}\bigg)=\sum_{i=0}^{d+1}(-1)^{d+1-i}\binom{d+1}{i}(2d+1-2i)^dx^{2d+1-2i}.
		\end{equation*}
		Our claim reduces now to show that this rational function is $0$ when computed in $x=1$. An easy induction shows that the left hand-side is always divisible by $(1-x^2)$ (and that $x=1$ does not kill the denominator). Therefore, when substituting $x=1$, the above quantity vanishes and the claim is proved: $B_0=2/d^d$.
		
		Combining the previous observations, we obtain
		\[
		\lim_{p\to\infty}\fs(A_p/f)= \frac{d^d}{2^{d+1}(d-1)!}\left(B_0+2B_1\right)= \frac{d^d}{2^{d+1}(d-1)!}\cdot\frac{4}{d^d}=\frac{1}{2^{d-1}(d-1)!}
		\]
		as required.
	\end{proof}

	
	\section{Explicit examples}\label{section:examples}
	In this last section, we provide explicit examples of $F$-signature functions of hypersurfaces.
	
	\subsection{Fermat cubic in four variables}
	
	We compute the $F$-signature function of the Fermat cubic in $4$ variables over a perfect field of characteristic $p>3$.
	
	\begin{theorem}\label{thm:FsignatureFermatcubic}
		Let $\KK$ be a perfect field of characteristic $p>3$, let $A_p=\KK\llbracket x_1,x_2,x_3,x_4\rrbracket$ and let $f=x_1^3+x_2^3+x_3^3+x_4^3\in A_p$. Then, the $F$-signature function of $f$ is, for any $e\geq1$,
		\[   \FS_{f}(e)= \fs(f)\cdot p^{3e}+\left(1-\fs(f)\right)\cdot \left(\frac{p+2\cdot (-1)^{\alpha}}{3}\right)^e,
		\]
		with $\displaystyle\fs(f)=\frac{3p(p-1)(p+1)}{8(3p^3-p+2\cdot(-1)^{\alpha+1}) }$ where $\alpha=\begin{cases}0 \ &\text{if } p\equiv1\ \mathrm{mod} \ 3 \\ 1&\text{if } p\equiv2\ \mathrm{mod} \ 3.
		\end{cases}$
	\end{theorem}
	\begin{proof}
		We start giving the details when $p\equiv1\ \mathrm{mod} \ 3$. Let $M=\lfloor \frac{p}{3}\rfloor=\frac{p-1}{3}$. In this setting, $M$ is always even. The shifting rule \eqref{eq:shifting rule p congruo a 1} gives the relation
		\begin{equation*}
			pS(\mathscr{L}(\phi_{x^3,p}))=\lambda_M\mathscr{L}(\phi_{x^3,p})+3\sum_{i=0}^{M-1}(-1)^i\lambda_i\Delta.
		\end{equation*}
		Call $\u=\mathscr{L}(\phi_{x^3+y^3,p})$ and multiply the previous relation with itself one time. Using the properties of $\Delta$ (see Example \ref{ex:properties of Delta}), the multiplication rules in \eqref{eq:formule moltiplicazione lambda} and the first formula of Lemma \ref{lemma:chesaraimportante}, after a lenghty but elementary computation we find that
		\begin{equation*}
			p^2 S(\u)=\Big(\sum_{i=0}^{2M}\lambda_i\Big)\u+\Big(3\cdot\sum_{i=0}^{2M-1}(-1)^{i}\big(p-1-i-\Big\lfloor\frac{i}{2}\Big\rfloor\big)\lambda_i\Big)\Delta.
		\end{equation*}
		Let $\v=R(\u)$. Applying the $R$ operator to both sides and using the rules of equation \eqref{eq:rule between R and S} we obtain that
		\begin{equation*}
			p^2 S(\v)=\Big(\sum_{j=p-1-2M}^{p-1}\lambda_j\Big)\v-\Big(3\cdot\sum_{i=p-2M}^{p-1}(-1)^{i}(i-\Big\lfloor\frac{p-1-i}{2}\Big\rfloor)\lambda_i\Big)\Delta.
		\end{equation*}
		Point (b) of Proposition \ref{prop:properties of the function r} implies that
		\begin{equation*}
			r_4(\u,\v)=(1-p^{3}z)\cdot (-1)+z\cdot r_4\big(p^2S(\u),p^2S(\v)\big).
		\end{equation*}
		Substituting the appropriate shifting rules and using the properties of the function $r_4$ (see again Proposition \ref{prop:properties of the function r}) we find that
		\begin{equation}\label{eq:r4 in dimostrazione della morte}
			r_4(\u,\v)=\frac{-1+(p^3-C_p)z}{1-(M+1)z}
		\end{equation}
		where
		\begin{equation*}
			C_p=9\sum_{i=M+1}^{2M-1}\Big(p-1-i-\Big\lfloor\frac{i}{2}\Big\rfloor\Big)\Big(i-\Big\lfloor\frac{p-1-i}{2}\Big\rfloor\Big)+ 6\sum_{i=M+1}^{2M}(-1)^{i}(i-\Big\lfloor\frac{p-1-i}{2}\Big\rfloor).
		\end{equation*}
		Elementary considerations allow to simplify the last summand in the previous expression, obtaining
		\begin{equation*}
			6\sum_{i=M+1}^{2M}(-1)^{i}(i-\Big\lfloor\frac{p-1-i}{2}\Big\rfloor)=6\cdot \frac{M}{2}=3M=p-1,
		\end{equation*}
		and one can also explicitely compute the value of the first summand, obtaining
		\begin{equation*}
			9\sum_{i=M+1}^{2M-1}\Big(p-1-i-\Big\lfloor\frac{i}{2}\Big\rfloor\Big)\Big(i-\Big\lfloor\frac{p-1-i}{2}\Big\rfloor\Big)=\frac{p^3}{8}-\frac{9p}{8}+1.
		\end{equation*}
		This implies that $C_p=\frac{p^3}{8}-\frac{p}{8}$ and equation \eqref{eq:r4 in dimostrazione della morte} becomes
		\begin{equation*}
			r_4(\u,\v)=\frac{-1+(\frac{7p^3}{8}+\frac{p}{8})z}{1-(\frac{p+2}{3})z}
		\end{equation*}
		When, instead, $p\equiv2\ \mathrm{mod} \ 3$ we have that $M=\lfloor \frac{p}{3}\rfloor=\frac{p-2}{3}$ is always odd. Therefore, in view of the shifting rule \eqref{eq:shifting rule p congruo a -1}, one can carry out the same arguments as in the case $p\equiv1\ \mathrm{mod} \ 3$, with almost identical computations. We can merge the two cases by saying that if 
		\begin{equation*}
			\alpha=\begin{cases}0 \ &\text{if } p\equiv1\ \mathrm{mod} \ 3 \\ 1&\text{if } p\equiv2\ \mathrm{mod} \ 3
			\end{cases}
		\end{equation*}
		then, for every prime $p>3$, 
		\begin{equation*}
			r_4(\u,\v)=\frac{-1+(\frac{7p^3}{8}+\frac{p}{8})z}{1-(\frac{p+2\cdot (-1)^\alpha}{3})z}.
		\end{equation*}
		Applying Theorem~\ref{cor:from function r to FSS} we find that
		\begin{equation*}
			\FSS_{f}(z)=\frac{-1+(\frac{7p^3}{8}+\frac{p}{8})z}{(p^3z-1)(1-(\frac{p+2\cdot (-1)^\alpha}{3})z)},
		\end{equation*}
		which in turn implies our claim for the $F$-signature function, by applying Proposition~\ref{prop:successione-serie} and the partial fraction decomposition.
	\end{proof}
	
	\begin{example}
		The $F$-signature function of $f=x_1^3+x_2^3+x_3^3+x_4^3$ is
		\begin{equation*}
			\FS_{f}(e)=\frac{15}{124}5^{3e}+\frac{109}{124} \ \ \ \  \text{ for }p=5,  \quad \ \ \ \text{and} \ \ \ \quad \FS_{f}(e)=\frac{21}{170} 7^{3e} + \frac{149}{170} 3^e \ \ \ \ \text{for} \ p=7.
		\end{equation*}
		Notice that when $p=5$ the $F$-signature function is a polynomial in $p^e$, while for $p=7$ it is not polynomial or quasi-polynomial in $p^e$.
	\end{example}

	\subsection{Non-diagonal hypersurfaces}
	
	We conclude with two examples of $F$-signature which show how these methods can be applied also to non-diagonal hypersurfaces. We rely on some computations from \cite[Examples~5.1~and~5.2]{MonTexII}.
	
	\begin{example}\label{ex:nondiagonal1}
		Let $\KK$ be a perfect field of characteristic $p=3$ and let $f=y^3-x^4+x^2y^2+w^3-z^4+z^2w^2$. We shall calculate the $F$-signature function of $\KK\llbracket x,y,z,w\rrbracket/(f)$.
		In order to do so, we observe that we can write the hypersurface $f$ as $f(x,y,z,w)=g(x,y)+g(z,w)$, where $g(x,y)=y^3-x^4+x^2y^2$. Notice that $\phi_{g,3}$ is a $p$-fractal, since it is a two-variable polynomial (see \cite{MonTexI}).

		Set ${\bf a}=\mathscr{L}(\phi_{g,3})$ and ${\bf b}=\mathscr{L}(\bar{\phi}_{g,3})=R({\bf a})$.
		Then, by \cite[Example 5.1]{MonTexII}, we have the relations
		\begin{equation}\label{eq:examplenondiag1}
			9S({\a})={\b}+\lambda_1{\b}+9\Delta \quad\quad \text{and}\quad \quad 9S(\b)=\lambda_2 {\a}+\lambda_1 {\a} -9\lambda_2\Delta.
		\end{equation}
		For any $\u,\v\in\Lambda$, we consider the operator of Definition~\ref{dfn:operatore bilineare r} \[
		r_4(\u,\v)=(1-27z)\cdot\sum_{e=0}^\infty\alpha(u_ev_e)(81z)^e.
		\]
		Corollary~\ref{cor:from function r to FSS} yields $\FSS_{f}(z)=\FSS_{g+g}(z)=(27z-1)^{-1} r_4(\a,\b)$. 
		Using properties of Proposition~\ref{prop:properties of the function r} and the fact that $\alpha(a_0b_0)=-1$, we can compute
		\begin{equation*}
			\begin{split}
				r_4(\a,\b)&=(1-27z)\cdot\alpha(a_0b_0)+z\cdot r_4(9S(\a),9S(\b))=\\
				&=-1+27z+z\cdot r_4(\b+\lambda_1\b+9\Delta,\lambda_2\a+\lambda_1\a-9\lambda_2\Delta)=\\
				&=-1+27z+z\cdot r_4(\b,\a).
			\end{split}
		\end{equation*}
		This implies that $r_4(\a,\b)=\frac{27z-1}{1-z}$, hence $\FSS_{f}(z)=\frac{1}{1-z}$, that is 
		\begin{equation*}
			\FS_{f}(e)=1
		\end{equation*}
		for all $e\geq1$. In particular, the $F$-signature of $f$ is zero.
	\end{example}
	
	\begin{example}\label{ex:nondiagonal2}
		Let $\KK$ be a perfect field of characteristic $p=3$, and let $f=y^3-x^4+x^2y^2+zw(z+w)$. We shall calculate the $F$-signature function of $\KK\llbracket x,y,z,w\rrbracket/(f)$.
		We proceed in a similar way as in Example~\ref{ex:nondiagonal1} and we	write $f(x,y,z,w)=g(x,y)+h(z,w)$, where	
		$g(x,y)=y^3-x^4+x^2y^2$ and $h(z,w)=zw(z+w)$.
		We set $\a$ and $\b$ as in Example~\ref{ex:nondiagonal1}, and $\c=\mathscr{L}(\phi_{h,3})$. In addition to the relations \eqref{eq:examplenondiag1}, from \cite[Example 5.1]{MonTexII} we have 
		\[
		9S({\c})={\c} +\lambda_1 {\c} +(6\lambda_0-3\lambda_1)\Delta.
		\]
		So, in a similar way as before, we can compute
		\begin{equation*}
			r_4(\c,\a)=1+33z+2z\cdot r(\c,\b)\quad\quad\text{and}\quad\quad r_4(\c,\b)=-1+24z+z\cdot r(\c,\a),
		\end{equation*}
		which yield $r_4(\c,\b)=\frac{33z^2+25z-1}{1-2z^2}$. This implies that
		\begin{equation*}
			\FSS_{f}(z)=\FSS_{g+h}(z)=\frac{33z^2+25z-1}{(1-2z^2)(27z-1)}.
		\end{equation*}
		Thus, applying the partial fraction decomposition, the $F$-signature function of $f$ is
		\[
		\FS_{f}(e)=\frac{21}{727}3^{3e}+\frac{1412+887\sqrt{2}}{4\cdot 727}(\sqrt{2})^e+\frac{1412-887\sqrt{2}}{4\cdot 727}(-\sqrt{2})^e.
		\]
		In particular, the $F$-signature of $f$ is $\fs(f)=\frac{21}{727}$.
	\end{example}
	
	\begin{example}
		Let $\KK$ be a perfect field of characteristic $p=7$, and let $f=x^3+y^4+z^4+zw^3$. We compute the $F$-signature function of $\KK\llbracket x,y,z,w\rrbracket/(f)$. We write $f(x,y,z,w)=g(x,y)+h(z,w)$, where $g=x^3+y^4$ and $h=z^4+zw^3$. 
		Set $\u=\mathscr{L}(\phi_{g,7})$, $\a=\mathscr{L}(\phi_{h,7})$ and $\b=\mathscr{L}(\bar{\phi}_{h,7})$. For any $\v,\v'\in\Lambda$, we consider the operator of Definition~\ref{dfn:operatore bilineare r}
		\begin{equation*}
			r(\v,\v')=(1-7^3z)\cdot\sum_{e=0}^\infty\alpha(v_ev_e')(7^4z)^e.
		\end{equation*}
		From \cite[Examples~5.2]{MonTexII}, we have
		\begin{equation*}
			r_4(\u,\b)=-1+339z+z\cdot r(\u,\a)\quad\quad\text{and}\quad\quad r_4(\u,\a)=\frac{1+488z+679z^2+339z^3}{1-2z-z^3}.
		\end{equation*}
		This gives that
		\begin{equation*}
			\FSS_{f}(z)=\FSS_{g+h}(z)=\frac{-680z^3+190z^2-342z+1}{(343z-1)(z^3+2z-1)}.
		\end{equation*}
		Thus, the $F$-signature function of $f$ is
		\[
		\FS_{f}(e)=\frac{182139}{40118308}7^{3e}+A\alpha^e+B\beta^e+C\bar{\beta}^e,
		\]
		where $1/\alpha, 1/\beta$ and $1/\bar{\beta}$ are the complex roots of $z^3+2z-1$ and $A$, $B$ and $C$ are complex coefficients that can be explicitely computed pursuing the partial fraction decomposition of $\FSS_f(z)$. In particular, the $F$-signature of $f$ is $\fs(f)=\frac{182139}{40118308}$.
	\end{example}

	

	\bibliographystyle{siam}
	\bibliography{references}

@article{PSSY26,
      title={Hilbert-{K}unz multiplicity of quadrics via {E}hrhart theory}, 
      author={Igor Pak and Boris Shapiro and Ilya Smirnov and Ken-ichi Yoshida},
      year={2025},
      journal={arXiv:2508.17915},
}

@article{Meng26,
      title={Limits of {F}-invariants and {R}iemann-{S}tieltjes integral}, 
      author={Cheng Meng},
      year={2026},
      journal={arXiv:2507.13898}, 
}

@article{AE05,
 author               = {Aberbach, Ian and Enescu, Florian},
 doi                  = {0.1007/s00209-005-0776-y},
 fjournal             = {Mathematische Zeitschrift},
 issn                 = {0002-9947},
 journal              = {Math. Z.},
 pages                = {791--806},
 title                = {The structure of {$F$}-pure rings},
 volume               = {250},
 year                 = {2005},
 }

@article {SchwedeCentersofFpurity,
    AUTHOR = {Schwede, Karl},
     TITLE = {Centers of {$F$}-purity},
   JOURNAL = {Math. Z.},
  FJOURNAL = {Mathematische Zeitschrift},
    VOLUME = {265},
      YEAR = {2010},
    NUMBER = {3},
     PAGES = {687--714},
      ISSN = {0025-5874,1432-1823},
       URL = {https://doi.org/10.1007/s00209-009-0536-5},
}

@article{Fedder83,
 author               = {Fedder, Richard},
 coden                = {TAMTAM},
 doi                  = {10.2307/1999165},
 fjournal             = {Transactions of the American Mathematical Society},
 issn                 = {0002-9947},
 journal              = {Trans. Amer. Math. Soc.},
 mrclass              = {13H10 (13D03 14B05)},
 mrnumber             = {701505 (84h:13031)},
 mrreviewer           = {D. Kirby},
 number               = {2},
 pages                = {461--480},
 title                = {{$F$}-purity and rational singularity},
 url                  = {http://dx.doi.org.proxy.its.virginia.edu/10.2307/1999165},
 volume               = {278},
 year                 = {1983},
 }

@article{Glassbrenner96,
 author               = {Glassbrenner, Donna},
 coden                = {PAMYAR},
 doi                  = {10.1090/S0002-9939-96-03030-4},
 fjournal             = {Proceedings of the American Mathematical Society},
 issn                 = {0002-9939},
 journal              = {Proc. Amer. Math. Soc.},
 mrclass              = {13A35 (13H10)},
 mrnumber             = {1291770 (96d:13004)},
 mrreviewer           = {Ian M. Aberbach},
 number               = {2},
 pages                = {345--353},
 title                = {Strong {$F$}-regularity in images of regular rings},
 url                  = {http://dx.doi.org.proxy.lib.umich.edu/10.1090/S0002-9939-96-03030-4},
 volume               = {124},
 year                 = {1996},
 }

@article {CH98,
    AUTHOR = {Chiang, Li and Hung, Yu-Ching},
     TITLE = {On {H}ilbert-{K}unz functions of some hypersurfaces},
   JOURNAL = {J. Algebra},
  FJOURNAL = {Journal of Algebra},
    VOLUME = {199},
      YEAR = {1998},
    NUMBER = {2},
     PAGES = {499--527},
      ISSN = {0021-8693,1090-266X},
   MRCLASS = {13D40 (13A35 13H15)},
  MRNUMBER = {1489923},
MRREVIEWER = {J.\ K.\ Verma},
       DOI = {10.1006/jabr.1997.7206},
       URL = {https://doi.org/10.1006/jabr.1997.7206},
}

@article {MM23,
	AUTHOR = {Meng, Chen and Mukhopadhyay, Alapan},
	TITLE = {$h$-function, {H}ilbert-{K}unz density function and {F}robenius-{P}oincaré function},
	JOURNAL = {arXiv:2310.10270},
	YEAR = {2023},
}

@article {Ohta17,
    AUTHOR = {Ohta, Kosuke},
     TITLE = {A function determined by a hypersurface of positive
              characteristic},
   JOURNAL = {Tokyo J. Math.},
  FJOURNAL = {Tokyo Journal of Mathematics},
    VOLUME = {40},
      YEAR = {2017},
    NUMBER = {2},
     PAGES = {495--515},
      ISSN = {0387-3870},
   MRCLASS = {13F25 (13A35 13D40)},
  MRNUMBER = {3743730},
MRREVIEWER = {Irena\ Swanson},
       DOI = {10.3836/tjm/1502179239},
       URL = {https://doi.org/10.3836/tjm/1502179239},
}

@article {Trivedi23,
    AUTHOR = {Trivedi, Vijaylaxmi},
     TITLE = {The {H}ilbert-{K}unz density functions of quadric
              hypersurfaces},
   JOURNAL = {Adv. Math.},
  FJOURNAL = {Advances in Mathematics},
    VOLUME = {430},
      YEAR = {2023},
     PAGES = {Paper No. 109207, 63},
      ISSN = {0001-8708,1090-2082},
   MRCLASS = {13D40 (14J60)},
  MRNUMBER = {4617943},
       DOI = {10.1016/j.aim.2023.109207},
       URL = {https://doi.org/10.1016/j.aim.2023.109207},
}

@article {Trivedi18,
    AUTHOR = {Trivedi, V.},
     TITLE = {Hilbert-{K}unz density function and {H}ilbert-{K}unz
              multiplicity},
   JOURNAL = {Trans. Amer. Math. Soc.},
  FJOURNAL = {Transactions of the American Mathematical Society},
    VOLUME = {370},
      YEAR = {2018},
    NUMBER = {12},
     PAGES = {8403--8428},
      ISSN = {0002-9947,1088-6850},
   MRCLASS = {13D40 (13H15 14H60 14N05)},
  MRNUMBER = {3864381},
MRREVIEWER = {Aryampilly\ V.\ Jayanthan},
       DOI = {10.1090/tran/7268},
       URL = {https://doi.org/10.1090/tran/7268},
}

@article {Kun69,
    AUTHOR = {Kunz, Ernst},
     TITLE = {Characterizations of regular local rings of characteristic
              {$p$}},
   JOURNAL = {Amer. J. Math.},
  FJOURNAL = {American Journal of Mathematics},
    VOLUME = {91},
      YEAR = {1969},
     PAGES = {772--784},
      ISSN = {0002-9327,1080-6377},
   MRCLASS = {13.95},
  MRNUMBER = {252389},
MRREVIEWER = {M.\ Nagata},
       DOI = {10.2307/2373351},
       URL = {https://doi.org/10.2307/2373351},
}

@book {MunkresTopology,
    AUTHOR = {Munkres, James R.},
     TITLE = {Topology},
   EDITION = {Second},
 PUBLISHER = {Prentice Hall, Inc.},
   ADDRESS = {Upper Saddle River, NJ},
      YEAR = {2000},
     PAGES = {xvi+537},
      ISBN = {0-13-181629-2},
   MRCLASS = {54-01},
  MRNUMBER = {3728284},
}

@article {Kun76,
    AUTHOR = {Kunz, Ernst},
     TITLE = {On {N}oetherian rings of characteristic {$p$}},
   JOURNAL = {Amer. J. Math.},
  FJOURNAL = {American Journal of Mathematics},
    VOLUME = {98},
      YEAR = {1976},
    NUMBER = {4},
     PAGES = {999--1013},
      ISSN = {0002-9327,1080-6377},
   MRCLASS = {13E05 (14E15)},
  MRNUMBER = {432625},
MRREVIEWER = {Hamet\ Seydi},
       DOI = {10.2307/2374038},
       URL = {https://doi.org/10.2307/2374038},
}

@article {Sin05,
    AUTHOR = {Singh, Anurag K.},
     TITLE = {The {$F$}-signature of an affine semigroup ring},
   JOURNAL = {J. Pure Appl. Algebra},
  FJOURNAL = {Journal of Pure and Applied Algebra},
    VOLUME = {196},
      YEAR = {2005},
    NUMBER = {2-3},
     PAGES = {313--321},
      ISSN = {0022-4049,1873-1376},
   MRCLASS = {13A35},
  MRNUMBER = {2110527},
MRREVIEWER = {Ian\ M.\ Aberbach},
       DOI = {10.1016/j.jpaa.2004.08.001},
       URL = {https://doi.org/10.1016/j.jpaa.2004.08.001},
}

@article {BST13,
	AUTHOR = {Blickle, Manuel and Schwede, Karl and Tucker, Kevin},
	TITLE = {{$F$}-signature of pairs: continuity, {$p$}-fractals and
		minimal log discrepancies},
	JOURNAL = {J. Lond. Math. Soc. (2)},
	FJOURNAL = {Journal of the London Mathematical Society. Second Series},
	VOLUME = {87},
	YEAR = {2013},
	NUMBER = {3},
	PAGES = {802--818},
	ISSN = {0024-6107},
	MRCLASS = {13A35 (13D40 14B05 14F18)},
	MRNUMBER = {3073677},
	MRREVIEWER = {Adela N. Vraciu},
	DOI = {10.1112/jlms/jds070},
	URL = {https://doi.org/10.1112/jlms/jds070},
}

@article {CZ23,
	AUTHOR = {Caminata, Alessio and Zerman, Francesco},
	TITLE = {{H}ilbert-{K}unz series and weak $p$-fractals},
	JOURNAL = {arXiv:2510.19355},
	YEAR = {2025},
}

@article {AL03,
	AUTHOR = {Aberbach, Ian M. and Leuschke, Graham J.},
	TITLE = {The {$F$}-signature and strong {$F$}-regularity},
	JOURNAL = {Math. Res. Lett.},
	FJOURNAL = {Mathematical Research Letters},
	VOLUME = {10},
	YEAR = {2003},
	NUMBER = {1},
	PAGES = {51--56},
	ISSN = {1073-2780},
	MRCLASS = {13A35},
	MRNUMBER = {1960123},
	MRREVIEWER = {Karen E. Smith},
	DOI = {10.4310/MRL.2003.v10.n1.a6},
	URL = {https://doi.org/10.4310/MRL.2003.v10.n1.a6},
}

@article {Boun57,
	AUTHOR = {Bouniakowsky, Viktor},
	TITLE = {Sur les diviseurs numériques invariables des fonctions rationnelles entiéres},
	JOURNAL = {Mém. Acad. Sc. St. Pétersbourg},
	YEAR = {1857},
	NUMBER = {6},
	PAGES = {305-329},
}

@book {BH98,
	AUTHOR = {Bruns, Winfried and Herzog, J\"{u}rgen},
	TITLE = {Cohen-{M}acaulay rings},
	SERIES = {Cambridge Studies in Advanced Mathematics},
	VOLUME = {39},
	PUBLISHER = {Cambridge University Press},
   ADDRESS = {Cambridge},
	YEAR = {1993},
	PAGES = {xii+403},
	ISBN = {0-521-41068-1},
	MRCLASS = {13H10 (13-02)},
	MRNUMBER = {1251956},
	MRREVIEWER = {Matthew Miller},
}

@article {BuchC97,
	AUTHOR = {Buchweitz, Ragnar-Olaf and Chen, Qun},
	TITLE = {Hilbert-{K}unz functions of cubic curves and surfaces},
	JOURNAL = {J. Algebra},
	FJOURNAL = {Journal of Algebra},
	VOLUME = {197},
	YEAR = {1997},
	NUMBER = {1},
	PAGES = {246--267},
	ISSN = {0021-8693},
	MRCLASS = {13D40 (13A35 13C15)},
	MRNUMBER = {1480784},
	MRREVIEWER = {Elsa Zatini},
	DOI = {10.1006/jabr.1997.7060},
	URL = {https://doi.org/10.1006/jabr.1997.7060},
}

@article {CDS19,
	AUTHOR = {Caminata, Alessio and De Stefani, Alessandro},
	TITLE = {F-signature function of quotient singularities},
	JOURNAL = {J. Algebra},
	FJOURNAL = {Journal of Algebra},
	VOLUME = {523},
	YEAR = {2019},
	PAGES = {311--341},
	ISSN = {0021-8693},
	MRCLASS = {13A35 (14B05)},
	MRNUMBER = {3901707},
	MRREVIEWER = {Linquan Ma},
	DOI = {10.1016/j.jalgebra.2018.11.040},
	URL = {https://doi.org/10.1016/j.jalgebra.2018.11.040},
}

@article {HanMon,
	AUTHOR = {Han, C. and Monsky, P.},
	TITLE = {Some surprising {H}ilbert-{K}unz functions},
	JOURNAL = {Math. Z.},
	FJOURNAL = {Mathematische Zeitschrift},
	VOLUME = {214},
	YEAR = {1993},
	NUMBER = {1},
	PAGES = {119--135},
	ISSN = {0025-5874},
	MRCLASS = {13D40 (13J10)},
	MRNUMBER = {1234602},
	MRREVIEWER = {Elsa Zatini},
	DOI = {10.1007/BF02572395},
	URL = {https://doi.org/10.1007/BF02572395},
}

@article {Hoc22,
	AUTHOR = {Hochster, Melvin},
	TITLE = {Tight closure and strongly {F}-regular rings},
	JOURNAL = {Res. Math. Sci.},
	FJOURNAL = {Research in the Mathematical Sciences},
	VOLUME = {9},
	YEAR = {2022},
	NUMBER = {3},
	PAGES = {Paper No. 56, 26},
	ISSN = {2522-0144},
	MRCLASS = {13A35 (13C14 13D45 13F40 13H05 13H10)},
	MRNUMBER = {4475033},
	MRREVIEWER = {Dinh Thanh Trung},
	DOI = {10.1007/s40687-022-00353-z},
	URL = {https://doi.org/10.1007/s40687-022-00353-z},
}

@article {HH94,
	AUTHOR = {Hochster, Melvin and Huneke, Craig},
	TITLE = {Tight closure of parameter ideals and splitting in
		module-finite extensions},
	JOURNAL = {J. Algebraic Geom.},
	FJOURNAL = {Journal of Algebraic Geometry},
	VOLUME = {3},
	YEAR = {1994},
	NUMBER = {4},
	PAGES = {599--670},
	ISSN = {1056-3911},
	MRCLASS = {13A35 (13B21 13H10)},
	MRNUMBER = {1297848},
	MRREVIEWER = {Ian M. Aberbach},
}

@article {HR76,
	AUTHOR = {Hochster, Melvin and Roberts, Joel L.},
	TITLE = {The purity of the {F}robenius and local cohomology},
	JOURNAL = {Advances in Math.},
	FJOURNAL = {Advances in Mathematics},
	VOLUME = {21},
	YEAR = {1976},
	NUMBER = {2},
	PAGES = {117--172},
	ISSN = {0001-8708},
	MRCLASS = {14B15 (13C15 14M99)},
	MRNUMBER = {417172},
	MRREVIEWER = {Gerhard Pfister},
	DOI = {10.1016/0001-8708(76)90073-6},
	URL = {https://doi.org/10.1016/0001-8708(76)90073-6},
}

@article {HunLeu,
	AUTHOR = {Huneke, Craig and Leuschke, Graham J.},
	TITLE = {Two theorems about maximal {C}ohen-{M}acaulay modules},
	JOURNAL = {Math. Ann.},
	FJOURNAL = {Mathematische Annalen},
	VOLUME = {324},
	YEAR = {2002},
	NUMBER = {2},
	PAGES = {391--404},
	ISSN = {0025-5831},
	MRCLASS = {13C14 (13A35 13D07 13H10)},
	MRNUMBER = {1933863},
	MRREVIEWER = {\={I}. \={I}. Burban},
	DOI = {10.1007/s00208-002-0343-3},
	URL = {https://doi.org/10.1007/s00208-002-0343-3},
}

@book {Shideler13,
	AUTHOR = {Shideler, Samuel},
	TITLE = {The F-Signature of Diagonal Hypersurfaces},
	SERIES = {Senior Thesis, Department of Mathematics, Princeton University},
	YEAR = {2013}
}

@book {Shideler18,
	AUTHOR = {Shideler, Samuel},
	TITLE = {Limit F-Signature Functions of Diagonal Hypersurfaces},
	SERIES = {Ph.D. Thesis, Department of Mathematics, University of Illinois at Chicago},
	YEAR = {2018}
}

@incollection {BFS13,
    AUTHOR = {Benito, Ang\'{e}lica and Faber, Eleonore and Smith, Karen E.},
     TITLE = {Measuring singularities with {F}robenius: the basics},
 BOOKTITLE = {Commutative algebra},
     PAGES = {57--97},
 PUBLISHER = {Springer},
   ADDRESS = {New York},
      YEAR = {2013},
      ISBN = {978-1-4614-5291-1; 978-1-4614-5292-8},
   MRCLASS = {14B05},
  MRNUMBER = {3051371},
MRREVIEWER = {\c{S}erban\ B\u{a}rc\u{a}nescu},
       DOI = {10.1007/978-1-4614-5292-8\_3},
}

@article {MonAlg,
	AUTHOR = {Monsky, Paul},
	TITLE = {Algebraicity of some {H}ilbert--{K}unz multiplicities (modulo a conjecture)},
	JOURNAL = {arXiv:0907.2470},
	YEAR = {2009},
}

@article {GesselMonsky,
	AUTHOR = {Gessel, Ira M. and Monsky, Paul},
	TITLE = {The limit as $p\to\infty$ of the {H}ilbert-{K}unz multiplicity of $\sum x_i^{d_i}$},
	JOURNAL = {arXiv:1007.2004},
	YEAR = {2010},
}

@article {Mon83,
	AUTHOR = {Monsky, Paul},
	TITLE = {The {H}ilbert-{K}unz function},
	JOURNAL = {Math. Ann.},
	FJOURNAL = {Mathematische Annalen},
	VOLUME = {263},
	YEAR = {1983},
	NUMBER = {1},
	PAGES = {43--49},
	ISSN = {0025-5831},
	MRCLASS = {13C15 (13H99)},
	MRNUMBER = {697329},
	MRREVIEWER = {Hamet Seydi},
	DOI = {10.1007/BF01457082},
	URL = {https://doi.org/10.1007/BF01457082},
}

@article {MonTexI,
	AUTHOR = {Monsky, Paul and Teixeira, Pedro},
	TITLE = {{$p$}-fractals and power series. {I}. {S}ome 2 variable
		results},
	JOURNAL = {J. Algebra},
	FJOURNAL = {Journal of Algebra},
	VOLUME = {280},
	YEAR = {2004},
	NUMBER = {2},
	PAGES = {505--536},
	ISSN = {0021-8693},
	MRCLASS = {13D40 (13C40 13D45 13H10 28A80)},
	MRNUMBER = {2089250},
	MRREVIEWER = {Duong Qu\^{o}c Vi\cfudot{e}t},
	DOI = {10.1016/j.jalgebra.2004.05.016},
	URL = {https://doi.org/10.1016/j.jalgebra.2004.05.016},
}

@article {MonTexII,
	AUTHOR = {Monsky, Paul and Teixeira, Pedro},
	TITLE = {{$p$}-fractals and power series. {II}. {S}ome applications to
		{H}ilbert-{K}unz theory},
	JOURNAL = {J. Algebra},
	FJOURNAL = {Journal of Algebra},
	VOLUME = {304},
	YEAR = {2006},
	NUMBER = {1},
	PAGES = {237--255},
	ISSN = {0021-8693},
	MRCLASS = {13D40 (13C40 13D45 13H10 28A80)},
	MRNUMBER = {2256387},
	MRREVIEWER = {Duong Qu\^{o}c Vi\cfudot{e}t},
	DOI = {10.1016/j.jalgebra.2005.07.033},
	URL = {https://doi.org/10.1016/j.jalgebra.2005.07.033},
}

@book {EisenbudCommutativeAlgebra,
    AUTHOR = {Eisenbud, David},
     TITLE = {Commutative algebra},
    SERIES = {Graduate Texts in Mathematics},
    VOLUME = {150},
      NOTE = {With a view toward algebraic geometry},
 PUBLISHER = {Springer-Verlag},
   ADDRESS = {New York},
      YEAR = {1995},
     PAGES = {xvi+785},
      ISBN = {0-387-94268-8; 0-387-94269-6},
   MRCLASS = {13-01 (14A05)},
  MRNUMBER = {1322960},
MRREVIEWER = {Matthew\ Miller},
       DOI = {10.1007/978-1-4612-5350-1},
}

@article {PolTuc,
	AUTHOR = {Polstra, Thomas and Tucker, Kevin},
	TITLE = {{$F$}-signature and {H}ilbert-{K}unz multiplicity: a combined
		approach and comparison},
	JOURNAL = {Algebra Number Theory},
	FJOURNAL = {Algebra \& Number Theory},
	VOLUME = {12},
	YEAR = {2018},
	NUMBER = {1},
	PAGES = {61--97},
	ISSN = {1937-0652},
	MRCLASS = {13A35 (14B05)},
	MRNUMBER = {3781433},
	MRREVIEWER = {Ana Bravo},
	DOI = {10.2140/ant.2018.12.61},
	URL = {https://doi.org/10.2140/ant.2018.12.61},
}

@article {SVdB97,
	AUTHOR = {Smith, Karen E. and Van den Bergh, Michel},
	TITLE = {Simplicity of rings of differential operators in prime
		characteristic},
	JOURNAL = {Proc. London Math. Soc. (3)},
	FJOURNAL = {Proceedings of the London Mathematical Society. Third Series},
	VOLUME = {75},
	YEAR = {1997},
	NUMBER = {1},
	PAGES = {32--62},
	ISSN = {0024-6115},
	MRCLASS = {16S32},
	MRNUMBER = {1444312},
	MRREVIEWER = {Ian M. Musson},
	DOI = {10.1112/S0024611597000257},
	URL = {https://doi.org/10.1112/S0024611597000257},
}

@article {Tucker12,
	AUTHOR = {Tucker, Kevin},
	TITLE = {{$F$}-signature exists},
	JOURNAL = {Invent. Math.},
	FJOURNAL = {Inventiones Mathematicae},
	VOLUME = {190},
	YEAR = {2012},
	NUMBER = {3},
	PAGES = {743--765},
	ISSN = {0020-9910},
	MRCLASS = {13A35 (13D40 14B05)},
	MRNUMBER = {2995185},
	MRREVIEWER = {Alberto F. Boix},
	DOI = {10.1007/s00222-012-0389-0},
	URL = {https://doi.org/10.1007/s00222-012-0389-0},
}

@article {WatYos2000,
	AUTHOR = {Watanabe, Kei-ichi and Yoshida, Ken-ichi},
	TITLE = {Hilbert--{K}unz multiplicity and an inequality between
		multiplicity and colength},
	JOURNAL = {J. Algebra},
	FJOURNAL = {Journal of Algebra},
	VOLUME = {230},
	YEAR = {2000},
	NUMBER = {1},
	PAGES = {295--317},
	ISSN = {0021-8693},
	MRCLASS = {13H15 (13A35 13D40 13H10)},
	MRNUMBER = {1774769},
	MRREVIEWER = {Elsa Zatini},
	DOI = {10.1006/jabr.1999.7956},
	URL = {https://doi.org/10.1006/jabr.1999.7956},
}

@article {WatYos05,
	AUTHOR = {Watanabe, Kei-ichi and Yoshida, Ken-ichi},
	TITLE = {Hilbert--{K}unz multiplicity of three-dimensional local rings},
	JOURNAL = {Nagoya Math. J.},
	FJOURNAL = {Nagoya Mathematical Journal},
	VOLUME = {177},
	YEAR = {2005},
	PAGES = {47--75},
	ISSN = {0027-7630},
	MRCLASS = {13D40 (13A35)},
	MRNUMBER = {2124547},
	MRREVIEWER = {Florian Enescu},
	DOI = {10.1017/S0027763000009053},
	URL = {https://doi.org/10.1017/S0027763000009053},
}

@article {WatYosh,
	AUTHOR = {Watanabe, Kei-ichi and Yoshida, Ken-ichi},
	TITLE = {Minimal relative {H}ilbert--{K}unz multiplicity},
	JOURNAL = {Illinois J. Math.},
	FJOURNAL = {Illinois Journal of Mathematics},
	VOLUME = {48},
	YEAR = {2004},
	NUMBER = {1},
	PAGES = {273--294},
	ISSN = {0019-2082},
	MRCLASS = {13D40 (13A35 13H10 13H15)},
	MRNUMBER = {2048225},
	MRREVIEWER = {Keri Sather-Wagstaff},
	URL = {http://projecteuclid.org/euclid.ijm/1258136184},
}

\end{document}